\documentclass[reqno]{amsart}

\usepackage{amssymb}
\usepackage{amsmath}
\usepackage{amsthm}
\usepackage{amsbsy}
\usepackage{bm}
\usepackage{hyperref}
\date{\today}
\usepackage{cite}
\usepackage{array}
\usepackage{color}
\usepackage{graphicx}
\usepackage[T1]{fontenc}
\usepackage[utf8]{inputenc}

\usepackage{geometry}
\theoremstyle{theorem}
    \newtheorem{thm}{Theorem}[section]
    \newtheorem{theorem}{Theorem}
    \newtheorem{lemma}[thm]{Lemma}
    
    \newtheorem{corollary}[thm]{Corollary}

\theoremstyle{definition} 

    \newtheorem{remark}[thm]{Remark}
    \newtheorem{example}[theorem]{Example}
    \newtheorem{exercise}[theorem]{Exercise}

\def\eps{\varepsilon}

\def\smallhalf{\mbox{ $\frac{1}{2}$}}

\def \K{{\mathcal K}}

\def\G{{\mathcal G}}

\def\N{\mathbb{N}}

\def\R{\mathbb{R}}

\def\l{\left}
\def\r{\right}
\def\<{\langle}
\def\>{\rangle}

\newcommand{\one}{{\mathbf 1}}
\newcommand{\E}{\mathbb{E}}

\def\P{{\bf P}}

\def\given{\left.\vphantom{\hbox{\Large (}}\right|}

\def\cL{\mathcal{L}}
\def\cX{\mathcal{X}}

\newcommand{\erdren}{Erd\H{o}s-R\'enyi }
\newcommand{\limninf}{\lim_{n\to \infty}}

\newcommand\mnote[1]{} 
\newcommand\be{\begin{equation*}}

\newcommand\ee{\end{equation*}}

\newcommand\ben{\begin{equation}}
\newcommand\een{\end{equation}}
\newcommand\bes{\begin{eqnarray*}}
\newcommand\ees{\end{eqnarray*}}

\newcommand\bex{\begin{exercise}}
\newcommand\eex{\end{exercise}}
\newcommand\beg{\begin{example}}
\newcommand\eeg{\end{example}}
\newcommand\benu{\begin{enumerate}}
\newcommand\eenu{\end{enumerate}}
\newcommand\beit{\begin{itemize}}
\newcommand\eeit{\end{itemize}}
\newcommand\berk{\begin{remark}}
\newcommand\eerk{\end{remark}}
\newcommand\bdefn{\begin{defintion}}
\newcommand\edefn{\end{definition}}
\newcommand\bthm{\begin{theorem}}
\newcommand\ethm{\end{theorem}}
\newcommand\bprf{\begin{proof}}
\newcommand\eprf{\end{proof}}
\newcommand\blem{\begin{lemma}}
\newcommand\elem{\end{lemma}}

\newcommand{\sm}{{\raise0.3ex\hbox{$\scriptstyle \setminus$}}}

\def\bi{\mathbf{i}}
\def\bj{\mathbf{j}}
\def\bp{\mathbf{p}}

\def\bm{\mathbf{m}}

\def\l{\left}
\def\r{\right}

\def\eps{\epsilon}

\newcommand{\imply}{\;\;\;\Longrightarrow\;\;\;}

\def\CHI{\mathchoice%
{\raise2pt\hbox{$\chi$}}%
{\raise2pt\hbox{$\chi$}}%
{\raise1.3pt\hbox{$\scriptstyle\chi$}}%
{\raise0.8pt\hbox{$\scriptscriptstyle\chi$}}}
\def\smalloplus{\raise1pt\hbox{$\,\scriptstyle \oplus\;$}}

\DeclareMathOperator{\spec}{spec}

\title{On the Spectrum of Dense Random Geometric Graphs}

\author{Kartick Adhikari}
\address{Department of Electrical Engineering, Technion - Israel Institute of Technology, Haifa, 3200003, Israel}

\email{kartickmath [at] gmail.com}

\author{Robert J. Adler}
\address{Department of Electrical Engineering, Technion - Israel Institute of Technology, Haifa, 3200003, Israel}

\email{radler [at] technion.ac.il}

\author{Omer Bobrowski}
\address{Department of Electrical Engineering, Technion - Israel Institute of Technology, Haifa, 3200003, Israel}

\email{omer [at] ee.technion.ac.il}

\author{Ron Rosenthal}
\address{Department of Mathematics, Technion - Israel Institute of Technology, Haifa, 3200003, Israel}

\email{ron.ro [at] technion.ac.il}

\thanks{(1) KA was supported in part by a Zeff Fellowship, a Viterbi Fellowship and the Israel Science Foundation, Grants 2539/17 and  771/17.}
\thanks{(2)  RJA  was supported in part by  the Israel Science Foundation, Grant 2539/17}
\thanks{(3) OB was supported in part by the Israel Science Foundation, Grant 1965/19}
\thanks{(4)  RR was supported in part by the Israel Science Foundation, Grant  771/17 and the US-Israel Binational Science Foundation, Grant 2018330}

\date{\today}

\begin{document}
	
	\begin{abstract}
		 In this paper we study the spectrum of the random geometric graph $G(n,r)$, in a regime where the graph is dense and highly connected. In the \erdren $G(n,p)$ random graph it is well known that upon connectivity the spectrum of the normalized graph Laplacian is concentrated around $1$. We show that such concentration does not occur in the $G(n,r)$ case, even when the graph is dense and almost a complete graph. In particular, we show that the limiting spectral gap is strictly smaller than $1$. In the special case where the vertices are  distributed uniformly in the unit cube and $r=1$, we show that for every $0\le k \le d$ there are at least $\binom{d}{k}$ eigenvalues near $1-2^{-k}$, and the limiting spectral gap is exactly $1/2$. We also show that the corresponding eigenfunctions in this case are tightly related to the geometric configuration of the points. 
	\end{abstract}
	
	\maketitle



	\section{Introduction}

	Let $G$ be an undirected graph on the vertex set $[n]:=\{1,2,\ldots,n\}$, and let $A$ be its adjacency matrix. The degree of the vertex $i$ is then $d_i = \sum_{j\ne i} A_{i,j}$, and the {\it graph Laplacian} is defined as $L := D-A$, where $D$ is the diagonal matrix with $d_1,\ldots,d_n$ on the  diagonal. The {\it symmetrically normalized graph Laplacian} is defined as
	\begin{equation}\label{eqn:lap}
	\cL := D^{-\smallhalf} L D^{-\smallhalf} = I - D^{-\smallhalf} A D^{-\smallhalf}\,.
	\end{equation}
	We are interested in the eigenvalues of $\cL$ denoted $\gamma_1 \le \gamma_2 \le \cdots \le \gamma_n$. 
	
	It is well known, cf. \cite{woess,chung_spectral_1997}, that $\gamma_i \in [0,2]$ for all $i$. In addition, $\gamma_1 = 0$, and the graph is connected if,  and only if, $\gamma_2 > 0$. The value of $\gamma_2$ is typically referred to as the {\it spectral gap} of the graph. 
	
		Graph Laplacians and  their  spectra contain important information about the connectivity structure of  graphs and the behavior of random walks on them, see for example \cite{Tan84, Dod84,AM85,AC88,BMG93,HLW06}. Graph spectra and harmonics also play key roles in various applications such as network analysis and machine learning \cite{shi_normalized_2000,von_luxburg_tutorial_2007,chung_spectral_1997}. 
	
	In this paper we study the spectrum of a {\it random geometric graph}.  Let $f:\R^d\to \R$ be a probability density function on $\R^d$, and let $\cX_n = \{X_1,\ldots, X_n\}\stackrel{i.i.d.}{\sim} f$. Let $\| \cdot \|$ be a  norm on $\R^d$. The random geometric graph $G(n,r)$ is  defined as the undirected graph with vertex set $[n]$, where $i$ is connected to $j$, abbreviated $i\sim j$, if, and only if,  $\| X_i - X_j\| \le r$.  That is, the entries of the adjacency matrix are of the form $A_{i,j} = h_r(X_i, X_j)$, where
	\begin{equation}\label{eqn:h_r}
	h_r(x,y) := \one_{\| x-y \| \le r}\,.
	\end{equation}
	Suppose that $f$ is uniform on  compact  $S\subset \R^d$. In that case, it can be shown \cite{penrosebook} that there exists a constant $C_S$ such that, if  $r = r(n) \ge C_S \left(\log n / n\right)^{1/d}$,  then, as $n\to \infty$, with high probability, $G(n,r)$ is connected.
	Since the graph is connected we know that $\gamma_2 > 0$. However, even if $r$ is much larger than $(\log n / n)^{1/d}$, as long as $r=r(n)\to 0$, using Cheeger's inequality it can be shown
	\cite{silva_spectral_2011} that  a.s.~$\gamma_2 \to 0$. This, in particular, implies that such graphs are not expanders, cf. \cite{chung_spectral_1997,HLW06}. From at least one aspect, this behavior is somewhat counter-intuitive, as $r$ can be chosen large enough so that the graph is $k$-connected with $k = k(n)\to \infty$ \cite{penrosebook}.
	
	This behavior is very different to that occurring in some other models of random graphs. In particular, for the  the \erdren random graph $G(n,p)$, it was shown \cite{FK81,FO05,Vu07,coja-oghlan_laplacian_2007,EKYY12,EKYY13,LvHY18,BBGK19,hoffman_spectral_2019} that above the connectivity threshold ($p = \log n /n$) the entire spectrum of the graph (except for $\gamma_1$) is concentrated around $1$, and, in particular that $\gamma_2 \to 1$. 
	
	In this paper we want to study a regime where the spectral gap of $G(n,r)$ is bounded away from zero. Thus, we have to take $r$ to be uniformly bounded away from zero, and, in particular, will take $r$ to be constant, independent of $n$. 
	We take $S$ to be the cube $[-1,1]^d$, equipped with the  $L^{\infty}$ norm,
	\begin{equation}\label{eqn:norm}
	\| x \| := \| x \|_\infty = \max_{1\le k \le d} |x_i|,\qquad  x = (x_1,\ldots,x_d)\in \R^d.
	\end{equation}
	Note that $[-1,1]^d$ is in fact the unit {\it ball} in the $L^\infty$ norm, and so we will denote it by $B^d$.
	
	The motivation behind these seemingly arbitrary choices ($L^\infty$ and $B^d$) is twofold. The first  is that of mathematical tractability, since these specific choices allow us to compute concrete estimates for the limiting  spectrum of $G(n,r)$ {which, for geometric reasons, would be much harder to compute with, for example, an Euclidean ball equipped with the $L^2$ metric.} The second reason for this choice arises from one of the key motivations for this work as a whole, which we now describe.\\
	
	\paragraph{\bf Homological connectivity in random Vietoris-Rips complexes.} A {\it simplicial complex} is a generalization of a graph, where in addition to vertices and edges,  it is possible to  include triangles, tetrahedra, and higher dimensional simplexes (finite subsets of vertices). Given a graph $G$, its corresponding {\it flag} (or {\it clique}) complex is constructed by adding a $k$-dimensional simplex (subsets of vertices of size $k+1$) for every $(k+1)$-clique in the graph. When $G = G(n,r)$ this complex is known as the random {\it Vietoris-Rips} (VR) complex. In  \cite{kahle_random_2011}, the homology groups (algebraic-topological structures describing cycles in various dimensions) of random VR-complexes were studied. 
	
	 One of the main open questions in this area is about the {\it homological connectivity} of these complexes. {In particular, one is interested in `phase transitions' (as $n$ and $r_n$ change) leading to geometric complexes for which one of these  homology groups} suddenly becomes trivial (in a suitable sense). This is  a higher-dimensional analogue of the traditional graph-connectivity property. This phenomenon was studied recently \cite{bobrowski_homological_2019} for a different type of a geometric complex,  known as the {\it random \v{C}ech complex}. The proof there, however, heavily relies on Morse theory, which is not applicable to the VR case. Therefore, a different approach is required.
	
	In  \cite{kahle_sharp_2014} Kahle studied homological connectivity in random flag complexes generated by the \erdren graph $G(n,p)$. Kahle's proof uses the so-called `Garland's method' \cite{garland_p-adic_1973}, that translates questions about the homology of a simplicial complex into questions about the graph Laplacian of its links. Combining Garland's method with  concentration results for the spectrum of $G(n,p)$ \cite{hoffman_spectral_2019}, leads to the proof of a phase-transition for homological connectivity. Garland's method was also used in the study of homological connectivity of other models of random simplicial complexes such as the high-dimensional \erdren model (the Linial-Mehuslam model), cf. \cite{gundert_eigenvalues_2016,ron-antti}.
	
	In the random VR complex, using scaling invariance, the relevant links can be shown to form random geometric graphs in the intersection of a finite number of unit balls. 
	Consequently,  {we believe that}  the analysis we provide here  for $S=B^d$ could be used to prove homological connectivity for the VR complex (in the $L^\infty$ norm). This  remains for future work.  	
	
	\section{Main results}
	
	Throughout the paper, $d \geq 1$ will denote dimension. Let $\cX_n = \{X_1,X_2,\ldots , X_n\}$ be i.i.d.~uniformly distributed random variables in  $[-1,1]^d$, and let $G(n,r)$ be the random geometric graph generated by $\cX_n$, using the $L^\infty$ norm \eqref{eqn:norm}, as described above. We will focus on $G_{n,r}:= G(n,r)$, for a fixed $r>0$ independent of $n$. 
	Define
	\begin{equation}\label{eqn:h_func}
	 h_{r}(x,y)=\one_{\| x-y \| \le r}\,,
	\end{equation}
	 let $A_{n,r}$ be the adjacency matrix of $G_{n,r}$, i.e.\ $(A_{n,r})_{i,j} = h_r(X_i,X_j)$, and let $\cL_{n,r}$ be the corresponding symmetrically normalized graph Laplacian \eqref{eqn:lap}. Finally, let 
	\begin{equation}\label{eqn:W_n}
	W_{n,r} = I_n - \cL_{n,r} = D_{n,r}^{-\smallhalf} A_{n,r} D_{n,r}^{-\smallhalf},
	\end{equation}
	where $I_n$ denotes the $n\times n$ identity matrix, and $D_{n,r}$ is the diagonal matrix of vertex degrees (i.e.~$D_{n,r} = \mathrm{diag}(d_1,\ldots, d_n)$ where $d_i = \sum_j (A_{n,r})_{i,j}$). 
	
	Let $\lambda_1^{(n)} \ge \lambda_2^{(n)}  \ge \cdots \ge \lambda_n^{(n)} $  be  the (ordered) eigenvalues of $W_{n,r}$, and define the empirical eigenvalue measure 
	\[
	\mu_n(\cdot) = \sum_{i=1}^n \delta_{\lambda_i^{(n)}}(\cdot),
	\]
	where $\delta_x$ denotes the Dirac delta measure on $\R$. Observe that,  for all $r>0$, $\lambda_1^{(n)}=1$, with corresponding eigenvector $(\sqrt{d_{1}},\ldots,\sqrt{d_{n}})$. 
	
\begin{remark} 
	When $2\le r<\infty$ {$G(n,r)$ is fully connected, and so} all  the entries of $A_{n,r}$ are equal to $1$. In this case, $\lambda_i^{(n)} = 0$ for all $i>1$,  so that $\mu_n=(n-1)\delta_0+\delta_1$, {and there is nothing interesting to study. Consequently, we will always assume that}  $0< r<2$. 
\end{remark}

{With basic notation out of the way, we can now summarise our main results, which  provide detailed information about the structure of the spectrum of $W_{n,r}$ (and consequently $\cL_{n,r}$), as well as some its harmonics, for large $n$.}

When $r=1$, we show that in  the limit (as $n\to\infty$) the spectrum of $W_{n,r}$ contains the values $1/2^k$ with multiplicity $\binom{d}{k}$ for all $0\le k \le d$. In addition, we show that the remaining $(n-2^d)$ eigenvalues are concentrated in the interval $(-0.3,0.3)$. 

When $r\in (1,2)$, we show that the entire spectrum (except for $\lambda_1^{(n)}$) is  contained in $(-1/2,1/2)$. 

Finally, when $r\in (0,1)$,  we show that the limit of $\lambda_2^{(n)}$ is larger than $1/2$. 

One consequence of these results is that the spectral gap of $\cL_{n,r}$ either converges to $1/2$ ($r=1)$, is strictly larger than $1/2$ ($r\in (1,2)$), or strictly smaller than $1/2$ ($r\in (0,1)$).
	
Here are the formal statements. The first result provides estimates for the case $r=1$.
	
	\begin{thm}\label{thm:limiteigenvalue}
		For $r=1$, the following holds almost surely.
		\begin{enumerate}	
			\item 
			For any $\delta > 0$, and $0\le k \le d$, define the open interval $I_{k,\delta} = (2^{-k}-\delta, 2^{-k}+\delta)$. Then, 
			\[
			\lim_{n\to \infty}\mu_n(I_{k,\delta}) \ge\binom{d}{k}.
			\]
\item Let $I \subset \mathbb{R}\backslash[-0.3,0.3]$, then
			\[
			\lim_{n\to \infty}\mu_n(I) = \delta_1(I) + d \delta_{1/2}(I).
			\]
		\end{enumerate}

	\end{thm}
	
	These results imply that, {for large enough (random) $n$, } the normalized Laplacian $\cL_{n,r}$ has at least $\binom{d}{k}$ eigenvalues around $1-2^{-k}$. Similarly, in the interval $(0,0.7)$ the only eigenvalues of  $\cL_{n,r}$ are $0$ and  $1/2$, and there are no eigenvalues in the interval $(1.3,2]$. 
%

The next two results  provide estimates for the cases where $r\neq 1$.
	\begin{thm}\label{thm:1to2}
Let $1<r<2$, then almost surely there exists $N>0$ such that  all $n\ge N$, 
	\[	|\lambda_k^{(n)}|< \smallhalf, 
	 \qquad \mbox{ for $k=2,\ldots, n$}.
	\] 
\end{thm}
\begin{thm}\label{thm:0to1}
Let $0<r<1$. Then, almost surely, there exists $N>0$ such that  for  all $n\ge N$, 
	\[		\lambda_2^{(n)}>   \smallhalf. 
	\]
\end{thm}
	
Theorems \ref{thm:limiteigenvalue}-\ref{thm:0to1} provide the following result about the spectral gap of the normalized Laplacian.

\begin{corollary}\label{cor}
Let $\cL_{n,r}$ be the normalized Laplacian of $G_{n,r}$, and recall that $\gamma_2^{(n)}$ is its spectral gap.
The following holds almost surely.
\begin{enumerate}
\item If $r=1$, then $\limninf \gamma_2^{(n)} = \smallhalf$.
\item If $r\in (1,2)$, then 
\begin{align}
	\frac{1}{2}< \liminf\gamma_2^{(n)}\le \limsup \gamma_2^{(n)}<1.
\end{align}.
\item If $r\in (0,1)$, then 
\begin{align}
0< \liminf\gamma_2^{(n)}\le \limsup \gamma_2^{(n)}<\frac{1}{2}.
\end{align}
\end{enumerate}
\end{corollary}

	As alluded to in the Introduction, this behavior is very different to that in the case  in the \erdren random graph $G(n,p)$. For, $G(n,p)$ we know \cite{hoffman_spectral_2019} that when the expected vertex-degree is a little above $\log n$, the spectrum of $\cL_{n,r}$ is concentrated around $1$. In particular, the spectral gap converges to $1$ (in probability). In the setting of $G(n,r)$, the graph is considerably denser, as the expected degree is
	proportional to $ n$, yet the spectral gap is much lower, and there is an entire sequence of eigenvalues between $0$ and $1$.
	
	
	The remainder of the paper is dedicated to proving Theorems \ref{thm:limiteigenvalue}-\ref{thm:0to1}  and their corollary.
	
	\section{Spectral convergence}
	
	The proofs of Theorems \ref{thm:limiteigenvalue}-\ref{thm:0to1}  rely heavily  on a suitable  definition convergence for the eigenvalues of a matrix. For this we exploit results from \cite{szegedy11} on the convergence of self-adjoint operators. In this section we provide the essential background.
	
	Let $(V,\nu)$ be a probability space, and denote by
	\begin{align*}
	\mathcal H := L^2(V,\nu)=\l\{ f:V\to \R \; : \; \int_V |f(x)|^2d\nu(x)<\infty \r\},
	\end{align*}
	the Hilbert space with the inner product 
	$$
		\langle f,g\rangle=\int_V f(x)g(x)d\nu(x),
	$$
	and associated norm $\|f\|_2:=\sqrt{\langle f,f\rangle}$.
	Let $K:V\times V\to \R$ be  a  kernel function in $L^2(V\times V, \nu\times \nu)$, and let  $\mathcal K\; :\;	\mathcal H \to 	\mathcal H $ be the {\it Hilbert-Schmidt kernel operator} for the kernel $K$,  defined by
	\begin{align*}
	\mathcal K f(x)=\int_V K(x,y)f(y)d\nu(y), \mbox{ for $f\in 	\mathcal H$}.
	\end{align*}
Since $K\in L^2(V\times V,\nu\times \nu)$, the operator $\mathcal K$ is compact, and hence its spectrum is given by a sequence of eigenvalues converging to zero. Furthermore, if $K(x,y)={K(y,x)}$, then the operator $\mathcal K$ is self-adjoint, and hence all of its eigenvalues are real. Throughout the paper we will use  $\spec(\cdot)$ to refer to the set of eigenvalues of a matrix or an operator, where eigenvalues are repeated according to their multiplicity.

The {\it cut norm} of $K$ is defined by
	\begin{align*}
	\|K\|_{\square}=\sup_{S,T}\l|\iint_{S\times T} K(x,y)d\nu(x) d\nu(y)\r|,
	\end{align*} 
	where $S, T$ run through all pairs of measurable sets in $V$. Note that 
	\begin{align*}
	\|K\|_{\square}\le \|K\|_1:=\iint_{V\times V}|K(x,y)|d\nu(x)d\nu(y).
	\end{align*}

	The following result is an extension of Lemma 1.11 in \cite{szegedy11}, and will be used in the proof of Theorem \ref{thm:limiteigenvalue}.
	\begin{lemma}[Lemma 1.11 in \cite{szegedy11} -- extended]\label{lem:szegedy}		Let $\{\mathcal K_n\}$ be a sequence of self-adjoint Hilbert-Schmidt kernel operators in $\mathcal H$ with corresponding kernels $\{K_n\}$,  such that $\sup_n\|K_n\|_{\infty}\le C$, for some $C>0$. 
		Suppose that  $K_n\to K$ in the cut norm. Let $\mathcal K$ be the Hilbert-Schmidt  kernel operator for the  kernel $K$. Then, for every $\lambda>0$ such that $\pm \lambda \not \in \spec(\K)$,
		\begin{equation}\label{eqn:szegedy}
		\begin{split}
		\lim_{n\to \infty} |\spec(\K_n)\cap (\lambda,\infty)| &=|\spec(\K)\cap (\lambda,\infty)|, 
		\\\lim_{n\to \infty} |\spec(\K_n)\cap (-\infty,-\lambda)|&=|\spec(\K)\cap (-\infty,-\lambda)|,
		\end{split}
		\end{equation}
		where $|\cdot |$ denotes cardinality.
	\end{lemma}
	
	The proof for Lemma \ref{lem:szegedy} is a modification of the proof in \cite{szegedy11}. In order to keep the paper self contained, we provide a proof in Appendix \ref{sec:appendix_spec}.
	
	\section{Outline for the proofs of Theorems \ref{thm:limiteigenvalue}, \ref{thm:1to2}, \ref{thm:0to1}}
	
	Before starting the proofs in detail, we will outline the main steps required for proving Theorems \ref{thm:limiteigenvalue}, \ref{thm:1to2} and  \ref{thm:0to1}. Lemma \ref{lem:szegedy} plays a key role in our proofs, where in our settings $V=B^d:=[-1,1]^d$ and $\nu_d$ is the uniform probability measure on $B^d$. The proof will then consist of three main steps.
	
	\vspace{.2cm}
	\noindent{\bf Step I :} We construct a special sequence of kernels $K_{n,r} : B^d\times B^d \to \R$ and show that the corresponding Hilbert-Schmidt operators satisfy $\spec(\mathcal K_{n,r})=\spec(W_{n,r})$.  This is achieved by using Lemma \ref{lem:equaleigenvalues} ($d=1$), Lemma \ref{lem:eigenvaluessame2} ($d= 2$), and Lemma \ref{lem:spec_d} ($d\ge 3$).
	
	\vspace{.2cm}
	\noindent{\bf Step II :} 
	Recall the definition of $h_r$ from  \eqref{eqn:h_func}, and define $K_r^d(x,y):B^d\times B^d\to \R$ by 
	\begin{align}\label{eqn:k_inf}
	 K_r^d(x,y):=\frac{h_r(x,y)}{\sqrt{H_r(x) H_r(y)}},
	\end{align}
	where $ 0<r<2$ and 
	\begin{equation}\label{eqn:H}
	\begin{aligned}
	 H_r(x):= \int_{B^d} h_r(x,u)d\nu_d(u) .
	 \end{aligned}
	\end{equation}
	We show that $\iint |K_r^d(x,y)|^2 d\nu_d(x)d\nu_d(y)<\infty$, {from which it follows that }
	the corresponding integral operator $\K_r^d$ is a Hilbert-Schmidt kernel operator. Furthermore, we show that 
	$\K^d_1$  has at least $\binom{d}{k}$ eigenvalues at $1/{2^k}$, for all $0\le k \le d$, and that the remaining eigenvalues lie in $(-0.3,0.3)$. This is the content of  Lemma \ref{lem:alleigenvalues}.  We also show that all the eigenvalues of $\K_r^d$ (except for $1$) lie in $(-0.5,0.5)$,  for all $1<r<2$. This is the content of Lemma \ref{lem:k_L}. On the other hand we show that the second largest eigenvalue of $\K_r^d$ is larger than $0.5$ for $0<r<1$. See Lemma \ref{lem:r<1}.

	\vspace{.2cm}
	\noindent{\bf Step III :} We show that $K_{n,r} \to K_r^d$ in the cut-norm, almost surely. This is carried out in Lemma \ref{lem:kernelconvergence} ($d=1$), Lemma \ref{lem:convergencekernel2} ($d= 2$), and Lemma \ref{lem:converge_d} ($d\ge 3$).

	\vspace{.2cm}
	Combining these three steps and Lemma \ref{lem:szegedy} gives us Theorems \ref{thm:limiteigenvalue}, \ref{thm:1to2} and \ref{thm:0to1}.
	
	The rest of the paper is organized as follows. In Section \ref{sec:eigenvalue} we calculate the eigenvalues of the limiting operator $\K_r^d$, for $0<r<2$. In Section \ref{sec:kernel_d1} we construct the kernels $K_{n,r}$ for $d=1$, and show their convergence. In Section \ref{sec:kernel_d2} we construct the kernels for $d=2$, and show their convergence. In Section \ref{sec:generald} we provide the details needed to generalize the two-dimensional case to arbitrary $d\geq 3$. Finally, the proofs for Theorems \ref{thm:limiteigenvalue}, \ref{thm:1to2}, \ref{thm:0to1} and Corollary \ref{cor} are given in Section \ref{sec:thm1}.
	
	To conclude this section, we present explicit formulae for $H_r(x)$ that will be useful for us later.
	For $0<r\leq 1$, we have that 
		\begin{equation}\label{eqn:H_r_01}
			H_r(x) = \begin{cases}
				\frac{1+r-|x|}{2} & 1-r \le |x| \le 1,\\
				r & |x|\leq 1-r.
			\end{cases}
		\end{equation}
For $1\leq r< 2$, we have
		\begin{equation}\label{eqn:H_r_12}
			H_r(x) = \begin{cases}
				\frac{1+r-|x|}{2} & r-1\leq |x|\leq 1, \\
				1 & |x|\leq r-1.
			\end{cases}
		\end{equation}
In particular, when $r=1$, we have $H_r(x) = 1-|x|/2$.

	\section{The spectrum of the limiting operators}\label{sec:eigenvalue}
	
	Recall the definition of the integral operator $\K_r^d$  with kernel $K_r^d$ given in \eqref{eqn:k_inf}. In this section we estimate the eigenvalues of this operator for arbitrary $d\geq 1$. First, we show that the operator is indeed  self-adjoint and Hilbert-Schmidt.
	\begin{lemma}	For every $d\geq 1$ the kernel $K_r^d$ satisfies $K_r^d(x,y)={K_r^d(y,x)}$ for all $x,y\in B^d$ and
	\[
		\iint_{B^d\times B^d}|K_r^d(x,y)|^2d\nu_d(x)d\nu_d(y)<\infty.
	\]
	That is, $\K_r^d$ is a self-adoint compact Hilbert-Schmidt operator. 
	\end{lemma}
	
	\begin{proof}
Note that, for $x=(x_1,\ldots,x_d)\in B^d$ and $y=(y_1,\ldots,y_d)\in B^d$, we have
		\begin{align*}
		h_r(x,y)=\prod_{i=1}^{d}h_r(x_i, y_i), \qquad\mbox{ and }\qquad H_r(x) = \prod_{i=1}^d H_r(x_i).
		\end{align*}
		Consequently,
		\begin{align}\label{eqn:kprod}
		K_r^d(x,y)=\prod_{i=1}^{d}K_r^1(x_i,y_i).
		\end{align}
		Thus, it suffices to prove the result for $K_r^1$. From the definition and the fact that $h_r(x,y)=h_r(y,x)$ for every $x,y\in B^1$, it follows that $K_r^1$ is real and $K_r^1(x,y)=K_r^1(y,x)$. Hence $K_r^1$ is symmetric.

Next, from \eqref{eqn:H_r_01} and \eqref{eqn:H_r_12}, for all $0<r<2$,  we have $H_r(x)\ge\frac{r}{2}>0$, and hence $K_r^1(x,y)\leq {2}/{r}$. The last equality then gives
\[
	\iint_{B^1\times B^1} |K_r^1(x,y)|^2d\nu_1(x)d\nu_1(y) \ \leq\  \frac{4}{r^2}\ <\ \infty\,,
\]
as required.
	\end{proof}

\subsection{General statements}
In this section we present a few lemmas that are true for all $r\in (0,2)$.
Denote by $(\lambda_i)_{i\ge 1}$ the eigenvalues of $W_{n,r}$ in decreasing order. Since we know that the spectrum of $\cL_{n,r}$ is in $[0,2]$, we have that the spectrum of $W_{n,r}$ is in $[-1,1]$. In the following sections we provide the proofs for our estimates of the spectrum for  different values of $r$. The first eigenvalue, however, is the same for all $r\in (0,2)$.

	\begin{lemma}\label{lem:largestev}
	Let $\mathcal K_r^1$ be as defined above. Then $\lambda_1 = 1$ and $\sqrt{H_r(x)}$ is the corresponding eigenfunction.
	\end{lemma}
	
	\begin{proof}
		Let $f(x)=\sqrt{H_r(x)}$. Then 
		\begin{align*}
		\mathcal K_r^1f(x)=\int_{B^1} \frac{h_r(x,y)}{\sqrt{H_r(x)H_r(y)}}\sqrt{H_r(y)}d\nu_1(y)=\frac{H_r(x)}{\sqrt{H_r(x)}}=f(x).
		\end{align*}
	\end{proof}
		The behavior of the remaining eigenvalues will be studied in the following sections, depending on the value of $r$.	

The next lemma shows that the space of eigenfunctions of $\mathcal K_r^1$ is spanned by a collection of even and odd functions.
		\begin{lemma}\label{lem:eigenfunction}
		Let  $\lambda $ be an  eigenvalue of $\mathcal K_r^1$ with corresponding eigenfunction $f$. Then ${f}^*(x):=f(-x)$ is also an eigenfunction with  the same eigenvalue $\lambda$. Consequently, both of the functions $f(x)+ f^*(x)$ and $f(x)- f^*(x)$ are also eigenfunctions corresponding to the eigenvalue $\lambda$, provided that they do not vanish.
	\end{lemma}
	
	\begin{proof}
		Let $\lambda$ be an eigenvalue of $\mathcal K_r^1$ with corresponding eigenfunction $f$. Recalling that $B^1 = [-1,1]$, we have 
		\begin{align}\label{eqn:integraleqn}
		\lambda f(x)=\int_{-1}^1 \frac{h_r(x,y)}{\sqrt{H_r(x)H_r(y)}}f(y)d\nu_1(y).
		\end{align}
		Note that
		\[
		H_r(-x)=\int_{-1}^1h_r(-x,y)d\nu_1(y)=\int_{-1}^1h_r(-x,-y)d\nu_1(y)=\int_{-1}^1h_r(x,y)d\nu_1(y)=H_r(x),
		\]
			where we used the change of variables $y\to (-y)$ and the fact that $h_r(-x,-y) = h_r(x,y)$. Therefore $H_r$ is even.
		Setting $f^*(x)=f(-x)$, we then have
		\begin{align*}
		\lambda  f^*(x)&= \lambda  f(-x)
		\\&=\int_{-1}^1 \frac{h_r(-x,y)}{\sqrt{H_r(-x)H_r(y)}}f(y)d\nu_1(y)
		\\& =\int_{-1}^1 \frac{h_r(-x,-y)}{\sqrt{H_r(-x)H_r(-y)}}f(-y)d\nu_1(y)
		\\&=\int_{-1}^1 \frac{h_r(x,y)}{\sqrt{H_r(x)H_r(y)}}f^*(y)d\nu_1(y),
		\end{align*}
		where, as before, we used the change of variables $y\to (-y)$ and the fact that $h_r(-x,-y) = h_r(x,y)$, as well as the fact that $H_r$ is even.
	\end{proof}


To prove some of our statements below,
we will need to use an auxiliary kernel
\begin{equation}\label{eqn:aux_ker}
K_r'(x,y)=\frac{h_r(x,y)}{H_r(x)},
\end{equation}
and $\mathcal K_r':L^2([-1,1],\nu_1)\to L^2([-1,1],\nu_1)$ the Hilbert-Schmidt kernel operator associated with $K_r'$. The next lemma shows a spectral equivalence between $\K'_r$ and  $\K_r^1$.

\begin{lemma}\label{lem:sameeigenvalue}
	Let $0<r<2$. A function $f: B^1\to \R$ is an eigenfunction of $\K'_r$ with eigenvalue $\lambda$ if and only if $\sqrt{H_r(x)}f(x)$ is an eigenfunction of $\K_r^1$ with the same eigenvalue $\lambda$.
\end{lemma}

\begin{proof}
	A function $f$ is an eigenfunction of $\K_r'$ with eigenvalue $\lambda$ if, and only if, for every $x\in [-1,1]$
	\[
	\lambda H_r(x)f(x) = \int_{-1}^1 h_r(x,y) f(y)d\nu_1(y) = \int_{-1}^1 \frac{h_r(x,y)}{\sqrt{H_r(y)}} \sqrt{H_r(y)}f(y)d\nu_1(y).
	\] 
	In turn, the above is true if, and only if, for every $x\in [0,1]$ (recall that by \eqref{eqn:H_r_01} and \eqref{eqn:H_r_12} $H_r(x)>0$ for every $x\in [0,1]$)
	\[
	\lambda \sqrt{H_r(x)}f(x) =  \int_{-1}^1 \frac{h_r(x,y)\sqrt{H_r(y)}}{\sqrt{H_r(x)H_r(y)}} f(y)d\nu_1(y) = \int_{-1}^1 K_r^1(x,y) \sqrt{H_r(y)}f(y)d\nu_1(y),
	\] 
	which holds if, and only if, $\sqrt{H_r(x)} f(x)$ is an eigenfunction of $\K_r^1$ with eigenvalue $\lambda$. 
\end{proof}
In the next lemma we show that the eigenfunctions of  $\K_r^d$ are continuous.

\begin{lemma}\label{lem:continuous}
	Let $f$ be an eigenfunction of $\K_r^d$	corresponding to  a non-zero eigenvalue. Then $f$ is continuous on $[-1,1]$.
\end{lemma}
\begin{proof}
	It is enough to show that the result holds for $d=1$, due to the product form of  the eigenfunctions of $\K_r^d$.
	
	Let $f$ be an eigenfunction of $\K_r^1$ with corresponding, non-zero, eigenvalue $\lambda$ . Without loss of generality we assume that $\|f\|=1$. Therefore, we have
	\begin{align*}
	 f(x)=\frac{1}{\lambda}\int_{B^1} \frac{h_r(x,y)}{\sqrt{H_r(x)H_r(y)}}f(y)d\nu_1(y),
	\end{align*}
implying that	
\begin{equation}\label{eqn:con}
\begin{split}
		|f(x)-f(x')|&\le \frac{1}{|\lambda|}\int \l|\frac{h_r(x,y)}{\sqrt{H_r(x)H_r(y)}}-\frac{h_r(x',y)}{\sqrt{H_r(x')H_r(y)}}\r||f(y)|d\nu_1(y)
		\\ &\le \frac{2\sqrt 2}{|\lambda| r \sqrt r}\int \l|\sqrt {H_r(x')}h_r(x,y)-\sqrt{H_r(x)}h_r(x',y)\r||f(y)|d\nu_1(y).
	\end{split}
\end{equation}
	The last equation follows from the fact that $H_r(x)\ge r/2$. By the triangle inequality we have 
	\begin{eqnarray*}
		&&\l|\sqrt {H_r(x')}h_r(x,y)-\sqrt{H_r(x)}h_r(x',y)\r|\\
		&&\qquad\qquad \le \l|\sqrt{H_r(x')}-\sqrt{H_r(x)}\r|h_r(x,y)+\sqrt{H_r(x)}|h_r(x,y)-h_r(x',y)|.
	\end{eqnarray*}
	Note that $H_r(x)$ is continuous in $[-1,1]$. Therefore, for every $\eps>0$, there exists $\delta>0$ such that 
	\begin{align*}
		\l|\sqrt{H_r(x')}-\sqrt{H_r(x)}\r| \le \eps, \ \ \mbox{ if $|x-x'|<\delta$}.
	\end{align*}
	Note that $h_r(x,y)\le 1$ and $H_r(x)\le 1$, for $0<r<2$. Therefore, if $|x-x'|<\delta$, then
	\begin{align}\label{eqn:h11}
	\l|\sqrt {H_r(x')}h_r(x,y)-\sqrt{H_r(x)}h_r(x',y)\r|\le \eps + |h_r(x,y)-h_r(x',y)|.
	\end{align}
		Observe that, if $|x-x'|<\delta$, then 
	\begin{align*}
	|h_r(x,y)-h_r(x',y)|=\l\{\begin{array}{ll}
	1 & \mbox{ if $y\in (x\wedge x'+r,x\vee x'+r)\cup (x\wedge x'-r,x\vee x'-r)$,}
	\\
	\\ 0 & \mbox{ otherwise,}	\end{array}\r.
	\end{align*}
	where $x\wedge x'=\min\{x,x'\}$ and $x\vee x'=\max \{x,x'\}$. Thus $\nu_1(\{y\; :\; |h_r(x,y)-h_r(x',y)|=1\})\le \delta$.
	Therefore,  for  $|x-x'|<\delta$, by \eqref{eqn:h11}	and  the Cauchy-Schwarz inequality, we have
	\begin{align*}
		&\int \l|\sqrt {H_r(x')}h_r(x,y)-\sqrt{H_r(x)}h_r(x',y)\r||f(y)|d\nu_1(y)
		\\&\qquad\le \eps \int |f(y)|d\nu_1(y)+ \int |h_r(x,y)-h_r(x',y)| |f(y)|d\nu_1(y)
		\\&\qquad\le \eps \|f\|+\|f\|\sqrt{\int |h_r(x,y)-h_r(x',y)|^2d\nu_1(y)}\le \eps + \sqrt{\delta}.		
	\end{align*}
	The last inequality follows from the fact that $\nu_1(\{y\; :\; |h_r(x,y)-h_r(x',y)|=1\})\le \delta$ and $\|f\|=1$.
	Thus \eqref{eqn:con} implies the result, as $\eps, \delta>0$ can be chosen arbitrarily small.
\end{proof}

	\subsection{The spectrum of $\mathcal K_1^d$} The goal of this subsection is analyze the spectrum in the case $r=1$. 

	\begin{lemma}\label{lem:egv}
		Denote by $(\lambda_i)_{i\geq 1}$ the eigenvalues of $\K_1^1$, in decreasing order. Then $\lambda_2=1/2$ with matching eigenfunction $x\sqrt{H_1(x)}$, and for all $i\ge 3$ we have $\lambda_i\in (-0.3,0.3)$.
	\end{lemma}

	
	\begin{proof}
		Denote by $(\varphi_i)_{i\geq 1}$ the orthonormal eigenfunctions corresponding to the eigenvalues $(\lambda_i)_{i\ge 1}$. By the spectral theorem for self-adjoint, compact operators 
		\begin{align*}
		K_1^1(x,y)=\sum_{i=1}^{\infty}\lambda_i\varphi_i(x)\varphi_i(y),
		\end{align*}
		which implies that 
		\begin{align}\label{eqn1}
		\iint_{B^1\times B^1} |K_1^1(x,y)|^2d\nu_1(x)d\nu_1(y)=\sum_{i=1}^{\infty}\lambda_i^2.
		\end{align}
		On the one hand, for $r=1$, we have $H_1(x)=1-|x|/2$, and therefore 
		\begin{equation}\label{eqn2}
		\begin{split}
		&\iint_{B^1\times B^1} |K_1^1(x,y)|^2d\nu_1(x)d\nu_1(y)
		=\int_{-1}^1\int_{-1}^1 \frac{\one_{|x-y|\le 1}}{(2-|x|)(2-|y|)}dx dy
		\\=&\int_{-1}^0\int_{-1}^{1+x}  \frac{1}{(2-|x|)(2-|y|)}dx dy+\int_{0}^1\int_{x-1}^1 \frac{1}{(2-|x|)(2-|y|)}dx dy.
		\end{split}
		\end{equation}
		Noting that 
	\[	\int_{-1}^{1+x} \frac{1}{(2-|y|)} dy= 2\log 2-\log(1-x),\quad \text{and}\quad \int_{x-1}^{1} \frac{1}{(2-|y|)} dy=2\log 2-\log(1+x)\,,
		\]
		we conclude 
		\[	\iint_{B^1\times B^1} |K_1^1(x,y)|^2d\nu_1(x)d\nu_1(y)
		=4(\log 2)^2-2\int_{0}^{1}\frac{\log(1+x)}{2-x}dx \approx 1.33299.
		\]
    		Next, it is easy to see that $\sqrt{H_1(x)}$ and 
		$x\sqrt{H_1(x)}$
		are eigenfunctions of $\K_1^1$, with eigenvalues $\lambda_1 = 1$ and $\lambda_2=1/2$, respectively.  Therefore,
		\begin{align*}
		\sum_{i=3}^{\infty}\lambda_i^2\approx 0.0829<0.09\,,
		\end{align*}
		and so  $|\lambda_i|<0.3$ for all $i\ge 3$, as required.
	\end{proof}

	\begin{lemma}\label{lem:alleigenvalues}
		For every $0\leq k\leq d$ the operator $\K_1^d$ has eigenvalue ${1}/{2^k}$ with multiplicity at least $\binom{d}{k}$. Moreover, the rest of the eigenvalues lie in $(-0.3,0.3)$. 
	\end{lemma}

	\begin{proof}[Proof of Lemma \ref{lem:alleigenvalues}]
		Let $(\lambda_i)_{i\ge1}$ be the eigenvalues of $\mathcal K_1^1$ listed with multiplicities in decreasing order, and $(\varphi_i)_{i\ge1}$  the corresponding orthonormal eigenfunctions. Since $\mathcal K_1^1$ is  a compact and self-adjoint operator, the spectral theorem implies that $(\varphi_i)_{i\ge1}$ form an orthonormal basis of $L^2(B^1,\nu_1)$. Recall that $L^2(B^d,\nu_d)$ is the space of functions on $B^1\times \cdots \times B^1$ ($d$-times) with respect to the product measure $\nu_1\times \cdots \times \nu_1$ ($d$-times). Therefore $(\varphi_{i_1,\ldots,i_d})_{i_1,\ldots, i_d\in \N}$ is an orthonormal basis for $L^2(B^d, \nu_d)$, where  $\varphi_{i_1,\ldots,i_d}(x):=\varphi_{i_1}(x_1)\cdots\varphi_{i_d}(x_d)$ for all $x\in B^d$.

Using \eqref{eqn:kprod}, for $i_1,\ldots, i_d\in \N$  and $x\in  B^d$  we have
		\begin{align*}
		\mathcal K_1^d\varphi_{i_1,\ldots,i_d}(x)=\int K_1^d(x,y)\varphi_{i_1,\ldots,i_d}(y)d\nu_d(y)=\lambda_{i_1}\cdots \lambda_{i_d}\varphi_{i_1,\ldots,i_d}(x).
		\end{align*}
		Hence  $(\lambda_{i_1}\cdots \lambda_{i_d})_{i_1,\ldots,i_d\in \N}$ forms the complete list of eigenvalues of $\mathcal K_1^d$ including multiplicities.  In particular, by Lemma \ref{lem:egv}, if there exists $1\le k\le d$ such that $i_k\ge 3$ then 
		\begin{align*}
		|\lambda_{i_1}\cdots \lambda_{i_d}|<0.3.
		\end{align*}
		Lemma \ref{lem:egv} also implies that $\lambda_1=1$ and $\lambda_2=1/2$. Thus, by considering all the eigenvalues corresponding to $i_1,\ldots,i_d\in \{1,2\}$, we get $1/2^k$ as an eigenvalue of $\mathcal K_1^d$ with multiplicity at least $\binom{d}{k}$, for $k=0,\ldots, d$. This completes the proof.
	\end{proof}

	\subsection{The spectrum of $\mathcal K_r^d$ for $1<r<2$} 

	In this subsection we estimate the eigenvalues of $\mathcal K_r^d$ for $1<r<2$. Recall that $\mathcal K_r^d$ is the Hilbert-Schmidt kernel operator with the kernel $K_r^d$, as defined in \eqref{eqn:k_inf}. We start by analyzing the spectrum in the case $d=1$. 
	
	\begin{lemma}\label{lem:secondevforL}
		Let  $1<r<2$, and denote by $(\lambda_i)_{i\geq 1}$ the eigenvalues of $\mathcal K^1_r$ in decreasing order. Then 
		\[
		|\lambda_{i}|<\smallhalf, \qquad  i\ge 2.
		\]
	\end{lemma}
	
	We will show that the statement of Lemma \ref{lem:secondevforL} holds for even and odd eigenfunctions separately. 
Using Lemma \ref{lem:eigenfunction}, this will suffice to cover all the eigenfunctions.

\begin{lemma}\label{lem:odd}
	Let $1<r< 2$, and $\lambda$ be an eigenvalue $\mathcal K_r^1$ with an odd eigenfunction. Then $|\lambda|<\smallhalf$.
\end{lemma}

\begin{proof}
	Let $\mathcal S_{\tiny \mbox{odd}}=\{ f\in \mathcal S : f \text{ is odd and }\|f\|_2=1\}$, where $\mathcal S$ denotes the space of eigenfunctions of $\K_r^1$. Since $\lambda$ is an eigenvalue with an odd eigenfunction, 
\[
|\lambda|\le \sup \{|\langle \mathcal K_r^1 f,f \rangle| \; :\; {f \in \mathcal S_{\tiny \mbox{odd}}}\}.
\] 
For every $f\in \mathcal S_{\tiny \mbox{odd}}$
	\begin{align*}
	\langle \mathcal K_r^1 f,f \rangle&=\iint K_r^1(x,y)f(x)f(y)d\nu_1(x)d\nu_1(y)
	\\&=\int_{0}^1\int_{(x-r)\vee (-1)}^1\frac{f(x)f(y)}{\sqrt{H_r(x)H_r(y)}}d\nu_1(x)d\nu_1(y)
	\\&\qquad +\int_{-1}^0\int_{-1}^{(x+r)\wedge 1}\frac{f(x)f(y)}{\sqrt{H_r(x)H_r(y)}}d\nu_1(x)d\nu_1(y)
	\\&=2\int_{0}^1\int_{(x-r)\vee (-1)}^1\frac{f(x)f(y)}{\sqrt{H_r(x)H_r(y)}}d\nu_1(x)d\nu_1(y)\,,
	\end{align*}
where for the last equality we used a change of variables and the fact that $H_r$ is even and $f$ is odd.
Since for $f\in \mathcal S_{\tiny \mbox{odd}}$, we also have
\[
	\int_{(x-r)\vee (-1)}^{(r-x)\wedge 1}\frac{f(y)}{\sqrt{H_r(y)}}d\nu_1(y)=0\,,
\]
we conclude that 
	\[
	\langle \mathcal K_r^1 f,f \rangle=2\int_{0}^1\int_{(r-x)\wedge 1}^1\frac{f(x)f(y)}{\sqrt{H_r(x)H_r(y)}}d\nu_1(x)d\nu_1(y).
	\]
	Next, denote $\mathcal S_{\mbox{\tiny odd}}^+=\{f\in \mathcal S_{\tiny \mbox{odd}} ~:~ f(x)\ge 0 \mbox{ for } x\in [0,1]\}$. We claim that 
	\begin{align*}
	\sup_{f\in \mathcal S_{\tiny \mbox{odd}}}	|\langle \mathcal K_r^1 f,f \rangle|=\sup_{f\in \mathcal S_{\mbox{\tiny odd}}^+}\langle \mathcal K_r^1 f,f \rangle.
	\end{align*}
	Indeed, since $\mathcal S_{\mbox{\tiny odd}}^+\subset \mathcal S_{\mbox{\tiny odd}}$ the inequality $\geq $ holds trivially. As for the other direction, given $f\in \mathcal S_{\mbox{\tiny odd}}$ define $\hat f\in \mathcal S_{\mbox{\tiny odd}}^+$ by $\hat f(x)=|f(x)|$  for $x\in [0,1]$ and $\hat f(x)=-|f(x)|$ for $x\in [-1,0]$. Note that 
	\begin{align*}
	|\langle\mathcal K^1_r f,f\rangle| &= \l|2\int\limits_{0}^1\int\limits_{(r-x)\wedge 1}^1\frac{f(x)f(y)}{\sqrt{H_r(x)H_r(y)}}d\nu_1(x)d\nu_1(y)\r|\\
	& \le 2\int\limits_{0}^1\int\limits_{(r-x)\wedge 1}^1\frac{\hat f(x)\hat f(y)}{\sqrt{H_r(x)H_r(y)}}d\nu_1(x)d\nu_1(y)
	\\
	&=\langle \mathcal{K}_r \hat f,\hat f\rangle\,,
	\end{align*}
	where we used the fact that $1<r<2$, and hence that $r-x\geq 0$ for all $x\in [0,1]$.

Finally, note that, for $1<r<2$, we know that $[(r-x)\wedge 1,1]\subset [1-x,1]$ and also that, by \eqref{eqn:H_r_12}, $H_r(x)\geq H_1(x)$. Hence, for $f\in S_{\mbox{\tiny odd}}^+$
	\begin{equation}
	\begin{aligned}
	\langle	\mathcal{K}_r^1 f,f\rangle &=2\int_{0}^1\int_{(r-x)\wedge 1}^1\frac{f(x)f(y)}{\sqrt{H_r(x)H_r(y)}}d\nu_1(x)d\nu_1(y)\\
	& <2 \int_{0}^1\int_{1-x}^1\frac{f(x)f(y)}{\sqrt{H_1(x)H_1(y)}}d\nu_1(x)d\nu_1(y)\\
	&=\langle \mathcal{K}_1^1 f,f\rangle
	\\
	&\leq \smallhalf\,,
	\end{aligned}
	\end{equation}
where in the inequality we used Lemma \ref{lem:egv}. 
\end{proof}

\begin{lemma}\label{lem:even}
		Let $1<r<2$, and $\lambda$ be an eigenvalue $\mathcal K_r^1$ with  even eigenfunction $f$. If $f$ is orthogonal to the eigenfunction $\sqrt{H_r}$ (see Lemma \ref{lem:largestev}), then $|\lambda|< \smallhalf$. In particular, the multiplicity of the eigenvalue $1$ is one. 
\end{lemma}

\begin{proof}
	Let $f$ be an even eigenfunction with eigenvalue $\lambda$, which is orthogonal to $\sqrt{H_r}$ and define $g=f/\sqrt{H_r}$. By Claim \ref{lem:sameeigenvalue}, the function $g$ is an eigenfunction with eigenvalue $\lambda$ of $\K'_r$, and therefore, for $0\le x\le 1$, 
	\begin{align}\label{eqn:ev1}
	\lambda H_r(x)g(x)=\int\limits_{(x-r)\vee (-1)}^{1}g(y)d\nu_1(y)=\int\limits_{0}^{(r-x)\wedge 1}g(y)d\nu_1(y)+\int\limits_0^1 g(y)d\nu_1(y).
	\end{align}
	From the assumption that $f$ is orthogonal to $\sqrt{H_r}$ also know that 
	\begin{align}
	\int_{-1}^1 g(y)H_r(y)d\nu_1(y)=0.
	\end{align}
	Together with the assumption that $f$ is even and the fact that $H_r$ is even, it follows that
	\begin{align}\label{eqn:ev2}
		\int_0^1g(y)H_r(y)d\nu_1(y)=0\,.
	\end{align}

Define $\|f\|_\infty:=\sup\{|f(x)|~:~ 0\le x\le 1\}$. Lemma \ref{lem:continuous} implies that $f$ is continuous, and hence $g$ is a continuous, non-trivial, eigenfunction. Therefore, we can find $ x_0\in [0,1]$ such that $|g(x_0)|=\|g\|_\infty>0$.  Without loss of generality we assume that $g(x_0)>0$. The rest of the argument depends on the location of $x_0$. 
	
	\vspace{.2cm}
	\noindent{\bf Case 1 ($r-1\le x_0\le 1$).} 
	For $x\in [r-1,1]$,  $r-1\le r-x \le 1$. Using  \eqref{eqn:ev1} and  \eqref{eqn:ev2} we have that, 
	for all  $\alpha \in \R$,
	\begin{align*}
	2\lambda H_r(x)g(x)&=\int_0^{r-x}g(y)dy+\int_0^1g(y)dy
	\\&=\int_0^{r-x}g(y)dy+\int_0^1g(y)dy+\alpha \int_0^1g(y)H_r(y)dy\\
	&=\int_0^{r-1}g(y)(2+\alpha H_r(y))dy+\int_{r-1}^{r-x}g(y)(2+\alpha H_r(y))dy+\int_{r-x}^1 g(y)(1+\alpha H_r(y))dy.
	\end{align*}
	Choosing $\alpha=-2$ and using \eqref{eqn:H_r_12}, we obtain that $2+\alpha H_r(y)=0$ for $0\le y\le r-1$. Consequently, applying \eqref{eqn:H_r_12} again, for $x\in [r-1, 1]$,
	\begin{align}\label{eqn:eve3}
	2\lambda H_r(x)g(x)&=\int_{r-1}^{r-x}g(y)(2-2 H_r(y))dy+\int_{r-x}^1g(y)(1-2 H_r(y))dy\nonumber
	\\&=\int_{r-1}^{r-x}g(y)(2-(r+1-y))dy+\int_{r-x}^1g(y)(1-(r+1-y))dy\nonumber
	\\&=\int_{r-1}^{r-x}g(y)(y-(r-1))dy+\int_{r-x}^1g(y)(-r+y))dy\nonumber
	\\&\le \|g\|_\infty\int_{r-1}^{r-x} (y-(r-1))dy+\|g\|_\infty\int_{r-x}^1(r-y))dy\nonumber
	\\&=\frac{\|g\|_\infty}{2}(1-2x(1-x)-(r-1)^2)\le \frac{r}{2}(2-r)\|g\|_\infty,
	\end{align}
Taking $x=x_0$ in \eqref{eqn:eve3} gives
	\begin{align*}
	|\lambda||(r+1-x_0)|\|g\|_\infty=|\lambda||(r+1-x_0)||g(x_0)|\le \frac{r}{2}(2-r)\|g\|_\infty ,
	\end{align*}
	and, since $r+1-x_0\geq r$, that 
	\begin{align*}
		|\lambda|\le \frac{2-r}{2}<\smallhalf,
	\end{align*}
	where in the last inequality we used the assumption that $1<r<2$.
	
	\vspace{.2cm}
	\noindent{\bf Case 2 ($0\le x_0\le r-1$).}  
	Let $x\in [0,r-1]$. Then $r-x \ge 1$.
	Using  \eqref{eqn:ev1} and \eqref{eqn:ev2}, for every $\alpha \in \R$	we have
	\begin{align*}
	\lambda H_r(x)g(x)&=\int_0^1 g(y)dy+\alpha \int_0^1g(y)H_r(y)dy\\
	&=\int_0^{r-1}g(y)(1+\alpha H_r(y))dy+\int_{r-1}^1g(y)(1+\alpha H_r(y))dy.
	\end{align*}
	From \eqref{eqn:H_r_12} we have that $H_r(x)=1$ for $x\le r-1$. Thus, taking $\alpha=-1$, we have
	\begin{align*}
	\lambda g(x)&=\int_{r-1}^1g(y)(1- H_r(y))dy\\&=\smallhalf\int_{r-1}^1g(y)(y-(r-1))dy
	\\&\le \frac{\|g\|_\infty}{2}\int_{r-1}^1(y-(r-1))dy
	\\&=\frac{\|g\|_\infty}{4}(2-r)^2.
	\end{align*}
	In particular, for $x=x_0$,
	\begin{align*}
	|\lambda|\cdot \|g\|_\infty =|\lambda|g(x_0)\le \frac{\|g\|_\infty}{4}(2-r)^2 \leq \frac{1}{4}\|g\|_\infty\,,
	\end{align*}
	and hence $|\lambda|\leq \frac{1}{4}$.
\end{proof}

Finally, we can prove Lemma \ref{lem:secondevforL}.
\begin{proof}[Proof of Lemma \ref{lem:secondevforL}]
	Note that Lemma \ref{lem:eigenfunction} implies that the eigenfunctions are generated by only odd and even functions. Thus, combining Lemmas \ref{lem:largestev}, \ref{lem:odd}, and \ref{lem:even}, and the fact that the eigenfunction form an orthonormal basis, the result follows.
\end{proof}

Now that we have estimates  for the spectrum in the one-dimensional case, we can  treat the case of  arbitrary dimension. 

\begin{lemma}\label{lem:k_L}
Let $1<r<2$. Then $\lambda=1$ is an eigenvalue of  $\K_r^d$ with multiplicity $1$, and all other eigenvalues $\lambda$ satisfy $|\lambda| < 1/2$.
\end{lemma}

\begin{proof}Fix $1<r<2$ and denote by $(\lambda_i)_{i\geq 1}$ the eigenvalues of $\K_r^1$ in decreasing order. Repeating the argument in the proof of Lemma  \ref{lem:alleigenvalues}, we see that the eigenvalues of $\K_r^d$ are $(\lambda_{i_1}\cdots \lambda_{i_d})_{i_1,\ldots,i_d\in \N}$. Hence the result follows from Lemma \ref{lem:secondevforL}.
\end{proof}


\subsection{The spectrum of $\mathcal K_r^d$ for  $0<r<1$} In this subsection we estimate the eigenvalues of $\mathcal K_r^d$ for $0<r<1$. In particular, we prove the following result. 

\begin{lemma}\label{lem:r<1}
Fix  $0<r<1$, and let $\mathcal K_r^d$ be as defined above. Then the second largest eigenvalue of $\mathcal K_r^d$ is strictly greater than $1/2$.
\end{lemma}

\begin{proof}
	Due to the product form of the eigenvalues of $\K_r^d$, it suffices to show that the second largest eigenvalue of $\K_r^1$ is strictly larger than $1/2$ for every $0<r<1$. 	
	
	Fix $0<r<1$ and denote by $\lambda_2$ the second largest eigenvalue of $\mathcal K_r^1$.  Since $\mathcal K_r^1$ is a self-adjoint operator, and $1$ is an eigenvalue with eigenfunction $\sqrt{H_r}$, 
	\begin{align*}
	\lambda_2=\sup_{f\in \mathcal S'} \frac{\langle \mathcal K^1_rf,f\rangle}{\langle f,f\rangle},
	\end{align*}
where $\mathcal S'=\{f\in L^2[-1,1] ~:~ f \mbox{ is orthogonal to }\sqrt{H_r}\}$. Hence, in order to prove the statement, it suffices to find a function $f\in \mathcal S'$ satisfying
	\[
	\frac{\langle \mathcal K_r^1f,f\rangle}{\langle f,f\rangle}>\smallhalf\,.
	\]
	Let $f:[-1,1]\to\mathbb{R}$ be given by $f(x)=x\sqrt{H_r(x)}$ for $x\in [-1,1]$. Since $\sqrt{H_r}$ is even and bounded, it follows that $f\in \mathcal S'$, as $f$ is odd. On  the one hand,  \eqref{eqn:H_r_01} implies that, for $0<r<1$,

	\begin{align}\label{eqn:fnorm}
		\langle f,f\rangle&=\int_{-1}^1|f(x)|^2d\nu_1(x)
		\\
		&=\int_{0}^{1}x^2H_r(x)dx\nonumber\\
		&=r\int_0^{1-r}x^2dx +\smallhalf\int_{1-r}^1x^2(1+r-x)dx.
	\end{align}
	On the other hand, using the definition of the kernel $\K_r^1$,
	\begin{align}\label{eqn:r<1}
	\langle \K_r^1f,f\rangle&=\iint K_r^1(x,y)f(x)f(y)d\nu_1(x)d\nu_1(y)\nonumber
	\\&=\frac{1}{4}\int_{-1}^{1}\int_{-1}^{1}xyh_r(x,y)dxdy\nonumber
	\\&=\smallhalf\int_{0}^{1}\int_{-1}^{1}xyh_r(x,y)dxdy\,.
	\end{align}
Recalling that $h_r(x,y)=\one_{|x-y|\le r}$ gives 
\begin{align}\label{eqn:2r<1}
	\int_{0}^{1}\int_{-1}^{1}xyh_r(x,y)dxdy&=\int_{0}^{1-r}\int_{x-r}^{x+r}xydydx+\int_{1-r}^1\int_{x-r}^1xydxdy\nonumber
	\\&=2r\int_{0}^{1-r}x^2dx+\smallhalf\int_{1-r}^1x(1-x^2+2xr-r^2)dx\nonumber
	\\&=\langle f,f\rangle+r\int_{0}^{1-r}x^2dx+\frac{(1-r)}{2}\int_{1-r}^1(1+r-x)dx\nonumber
	\\&\ge \langle f,f\rangle+\frac{r(1-r)^3}{3}+\frac{(1-r)r^2}{2}\,,
\end{align}
where the last equality follows from \eqref{eqn:fnorm}, and the last inequality follows from the fact that $1+r-x\ge r$. 
Combining \eqref{eqn:r<1} and \eqref{eqn:2r<1}, we conclude that for $0<r<1$
\[
\langle \K_r^1f,f\rangle > \smallhalf\langle f,f\rangle\,,
\]
as required.

\end{proof}

The following lemma will be used in the proof of  Corollary \ref{cor}.
\begin{lemma}\label{lem:lessthan1}
	Fix  $0<r<2$, and let $\mathcal K_r^d$ be as defined above.  Suppose $\lambda_2$ is the second largest eigenvalue of $\mathcal K_r^d$. Then   $0<\lambda_2< 1$.
\end{lemma}

\begin{proof}
		We first show that  $\lambda_2<1$. Due to the product form of the eigenvalues of $\K_r^d$,  it is enough to show that the result holds for $d=1$. 
		
		Let $\lambda $ be an eigenvalue of $\K_r^1$ with corresponding eigenfunction $f$, where $f$ is orthogonal to $\sqrt{H_r}$. It suffices  to show that $|\lambda|<1$, as eigenfunctions of the self-adjoint operator $\K_r^1$ are othogonal and  $\sqrt{H_r}$ is the eigenfunction for $\lambda=1$,  by Lemma \ref{lem:largestev}. Without loss of generality we assume that $f(x)=g(x)\sqrt{H_r(x)}$.
	  Since $\lambda$ is an eigenvalue of $\K_r^1$ with eigenfunction $f$, we have
	 \begin{align*}
	 	\lambda f(x)=\int_{-1}^1 K_r^1(x,y)f(y)d\nu_1(y).
	 \end{align*}
	 Consequently, as $H_r(x)>0$, we have that 
	 \begin{align}\label{eqn:glambda}
	 	\lambda \, g(x)=\frac{1}{H_r(x)}\int_{-1}^1 h_r(x,y)g(y)d\nu_1(y), \ \ \mbox{ for $x\in [-1,1]$}.
	 \end{align}
Since $f$ is orthogonal to $\sqrt{H_r}$, we also have that
	 \begin{align}\label{eqn:orthogonal}
	 \langle f, \sqrt{H_r}\rangle=\int_{-1}^1 g(x)H_r(x)d\nu_1(x)=0.
	 \end{align}
	  Observe that \eqref{eqn:orthogonal} implies that $g$ is a non-constant function in $[-1,1]$.  Also note that $g$ is continuous, as $f$ and $\sqrt{H_r}$ are  continuous by  Lemma \ref{lem:continuous} and \eqref{eqn:H_r_01}. Therefore, there exists  $x_0\in [-1,1]$ such that $g$ is not constant in the interval $[-1\vee (x_0-r), (x_0+r)\wedge 1]$ and $|g(x_0)|=\|g\|_{\infty}$, where $\|g\|_{\infty}:=\sup\{|g(x)|: -1\le x\le 1\}$. Thus, for $x=x_0$, from \eqref{eqn:glambda} we have that
	 \begin{align*}
	 	|\lambda| |g(x_0)|\le \frac{1}{H_r(x_0)}\int_{-1}^1 h_r(x_0,y)|g(y)|d\nu_1(y)<\frac{\|g\|_{\infty}}{H_r(x_0)}\int_{-1}^1 h_r(x_0,y)d\nu_1(y).
	 \end{align*}
	  The strict inequality in the last equation follows from the fact that $g$ is not  constant in the interval $[-1\vee (x_0-r), (x_0+r)\wedge 1]$.  Therefore 
	 \begin{align*}
	 |\lambda|<\frac{\|g\|_{\infty}}{|g(x_0)|}=1.
	 \end{align*}
	
	 Next, we show that $\lambda_2>0$. As before, it is enough to show the result holds for $d=1$. Since $\K_r^1$ is a self-adjoint operator and $\sqrt{H_r}$ is  the eigenfunction for the largest eigenvalue $1$,  we have
	 \begin{align*}
	 	\lambda_2=\sup_{f\in \mathcal S'}\frac{\langle \K_r^1f, f\rangle}{\langle f, f\rangle},
	 \end{align*}
	 where  $\mathcal S'=\{f\in L^2[-1,1] : \langle f, \sqrt{H_r}\rangle=0 \}$. 
	 Taking $f(x)=x\sqrt{H_r(x)}$, for $x\in [-1,1]$, then $\langle f, \sqrt{H_r}\rangle = 0$, and 
	 \begin{align*}
	 	\lambda_2\ge \frac{\langle \K_r^1f, f\rangle}{\langle f, f\rangle}>0.
	 \end{align*}
	 Hence the result.
\end{proof}
	\section{Construction of kernels for $d=1$}\label{sec:d1}	\label{sec:kernel_d1}
	
	In this section and those to follow we construct, for each $n$, a kernel whose spectrum is the same as the spectrum of the symmetrically normalized adjacency operator of the random geometric graph $G(n,r)$, and show that they converge in the cut norm to the limiting integral operators $\K_r^d$ of the previous section. This section is devoted to the proof in the case $d=1$ and the following sections are dedicated to the cases $d=2$ and general $d$. The main reason for this partition is pedagogical, as we wish to present the proofs in an incremental level of difficulty, allowing each step to rely on the preceding ones. 
	
	Throughout this section fix $0<r<2$. We start by defining a partition of $[-1,1]$ into subintervals. For $n\geq 1$, define $(L_i^n)_{i=1}^n$ by 
	\begin{equation}\label{eqn:L_i}
	\begin{split}
	L_i^n&= \l[-1+\frac{2(i-1)}{n}, -1+\frac{2i}{n}\r),\qquad 1\leq i\leq n-1,\\
	L_n^n &= \l[1-\frac{2}{n}, 1\r],
	\end{split}
	\end{equation}
	so that  the $L_i^n$ are disjoint intervals, with $\nu_1(L_i^n) = 1/n$ for all $i$. For brevity, throughout this section, we write $L_i$ for $L_i^n$.
	
	Let $X_1,\ldots, X_n$ be a sequence of i.i.d.~uniformly distributed random variables in $B^1 = [-1,1]$, and let $X^{(1)},\ldots, X^{(n)}$ be their order statistics, i.e., $X^{(1)}\le \cdots\le X^{(n)}$.  For $n\geq 1$, define  the random functions $h_{n,r}  :  B^1\times B^1\to \R$, as
	\begin{align*}
	h_{n,r}(x,y)=\sum_{i,j=1}^{n}h_r(X^{(i)},X^{(j)})\one_{L_i}(x)\one_{L_j}(y).
	\end{align*}
	
	Next, define a sequence of random kernels $K_{n,r  }:B^1\times B^1\to \R$ by
	\begin{align*}
	K_{n,r}(x,y)=\frac{h_{n,r}(x,y)}{\sqrt{H_{n,r}(x)H_{n,r}(y)}},
	\end{align*}
	where 
	\begin{equation}\label{eqn:H_n}
	H_{n,r}(x) := \int_{B^1} h_{n,r}(x,u)d\nu_1(u) = \frac{1}{n}\sum_{i=1}^n \one_{L_i}(x)\sum_{p=1}^n h_r  (X^{(i)}, X^{(p)}).
	\end{equation}
	Hence,
	\begin{equation}\label{eqn:K_n_formula}
	K_{n,r}(x,y)=\sum_{i,j=1}^{n}\frac{nh_r(X^{(i)},X^{(j)})\one_{L_i}(x)\one_{L_j}(y)}{\sqrt{\sum_{p,q=1}^{n}h_r(X^{(i)},X^{(p)})h_r(X^{(j)}, X^{(q)})}}.
	\end{equation}
	Let $\mathcal K_{n,r} : L^2(B^1,\nu_1)\to L^2(B^1,\nu_1)$ be the Hilbert-Schmidt kernel operator for the kernel $K_{n,r}$, i.e.
	\begin{equation}\label{eqn:K_n}
	\mathcal K_{n,r}f(x)=\int_{B^1} K_{n,r}(x,y)f(y)d\nu_1(y).
	\end{equation}
	Note that $\mathcal K_{n,r}$ is a {\it random} operator, as $K_{n,r}$ is a random function. As mentioned before, the goal of this section is to prove: (a) the operator $\mathcal  K_{n,r}$ has the same spectrum as the operator $W_{n,r}$, and (b) almost surely, $K_{n,r} \to K_r^1$ in the cut-norm as $n\to\infty$.
	
	\subsection{The spectrum of $\mathcal K_{n,r}$}
	
	We start by showing that $\mathcal K_{n,r}$ and $W_{n,r}$ have the same spectrum.
	\begin{lemma}\label{lem:equaleigenvalues}
		Let $d=1$ and $0<r<2$. For $n\geq 1$, let $\mathcal K_{n,r}$ be as defined in \eqref{eqn:K_n} and $W_{n,r}$ as defined in \eqref{eqn:W_n}. Then 
		$$\spec(\K_{n,r}) = \spec(W_{n,r}).$$
	\end{lemma}
	
	\begin{proof}
		Let $\widetilde A_{n,r}$ and $\widetilde W_{n,r}$ be the adjacency and symmetrically normalized adjacency matrices for the vertex set $X^{(1)},\ldots, X^{(n)}$. Since we only changed the order of the vertices,  $\spec(\widetilde W_{n,r}) = \spec(W_{n,r})$.
		Abbreviate $a_{i,j} = (\widetilde A_{n,r})_{i,j} = h_r(X^{(i)},X^{(j)})$, and $d_i = \sum_{j=1}^n a_{i,j}$. Then the $(i,j)$-th entry of $\widetilde W_{n,r}$ is given by 
		\begin{align*}
		w_{i,j} := \frac{a_{i,j}}{\sqrt{d_id_j}}.
		\end{align*}
		Using this notation with \eqref{eqn:K_n_formula}, we can also write
		\[
		K_{n,r}(x,y)= n\sum_{i,j=1}^{n}w_{i,j}\one_{L_i}(x)\one_{L_j}(y).
		\]
		Let $f$ be an eigenfunction of $\K_{n,r}$ with eigenvalue $\lambda$. Then, for every $x\in B^1$,
		\[
		\lambda f(x) = n\sum_{i=1}^n \left(\sum_{j=1}^n w_{i,j}\langle f, \one_{L_j}\rangle \right) \one_{L_i}(x).
		\]
		In other words, $f$ must be piecewise constant on the intervals $L_i$, and we can write
		\[
		f(x) = \sum_{i=1}^n c_i \one_{L_i}(x),
		\]
		for some values $c_i$. Hence
		\[
		\K_{n,r} f(x) = \sum_{i,j=1}^n  w_{i,j} c_j \one_{L_i}(x).
		\]
		Therefore, $f(x)$ is an eigenfunction of $\K_{n,r}$ with eigenvalue $\lambda$ if, and only if,
		\[
		\sum_{i=1}^n \left(\sum_{j=1}^n  w_{i,j} c_j \right) \one_{L_i}(x) = \lambda\sum_{i=1}^n c_i \one_{L_i}(x)\,.
		\]
		The last equation holds if, and only if, the vector $c = (c_1, \ldots, c_n)$ is an eigenvector of $\widetilde W_{n,r}$, with eigenvalue $\lambda$. That is, there is a one-to-one correspondence between eigenfunctions of $\K_{n,r}$ and the eigenvectors of $\widetilde W_{n,r}$, with matching eigenvalues. This concludes the proof.
	\end{proof}
	
	\subsection{Concentration of order statistics}
	The following lemma shows that the order statistics of the uniformly distributed random variables in $[-1,1]$ are concentrated around their means, which we will use to show that, almost surely, $\K_{n,r} \to \K_r^1$ in the cut-norm.

	\begin{lemma}\label{lem:concentration}
		Let $X^{(k)}$ be the order statistics of $n$ i.i.d.~uniformly distributed points in $B^1$. Then, almost surely, there exists $N>0$ such that, for all $n\ge N$, we have
		\begin{align*}
		\sup_{k=1,\ldots, n}\l|X^{(k)}-\E[X^{(k)}]\r|\le \frac{1}{n^{1/3}},
		\end{align*}
	where
	\begin{equation}\label{eq:The_expectation_of_X^(i)}
		\E[X^{(k)}] = -1 + \frac{2k}{n+1}\in L_{i}.
	\end{equation}
	\end{lemma}

The proof of Lemma \ref{lem:concentration} can be found in  \cite[Lemma 2]{fresen}. For completeness, we provide a simpler proof for this result in Appendix \ref{app1}.

	\subsection{The convergence of $K_{n,r}$}\label{sec:K_converge}
	In this subsection we show that $K_{n,r}$ converges to $K^{1}_r$ in the cut-norm almost surely as  $n\to \infty$. 	We start by defining, for $0<r<2$ and $\varepsilon>0$, the sets 
	\begin{equation}\label{eqn:G}
	\begin{split}
	\mathcal{ G}_1^\varepsilon&=\{(i,j) : |x-y|< r-\varepsilon, \mbox{ for all $(x,y)\in L_i\times L_j$} \},\\
	~\\
	\mathcal{ G}_2^\varepsilon&=\{(i,j): |x-y|>r+\varepsilon, \mbox{ for all $(x,y)\in L_i\times L_j$} \}.
	\end{split}
	\end{equation}
	
	\begin{figure}[h]
		\includegraphics[scale=0.3]{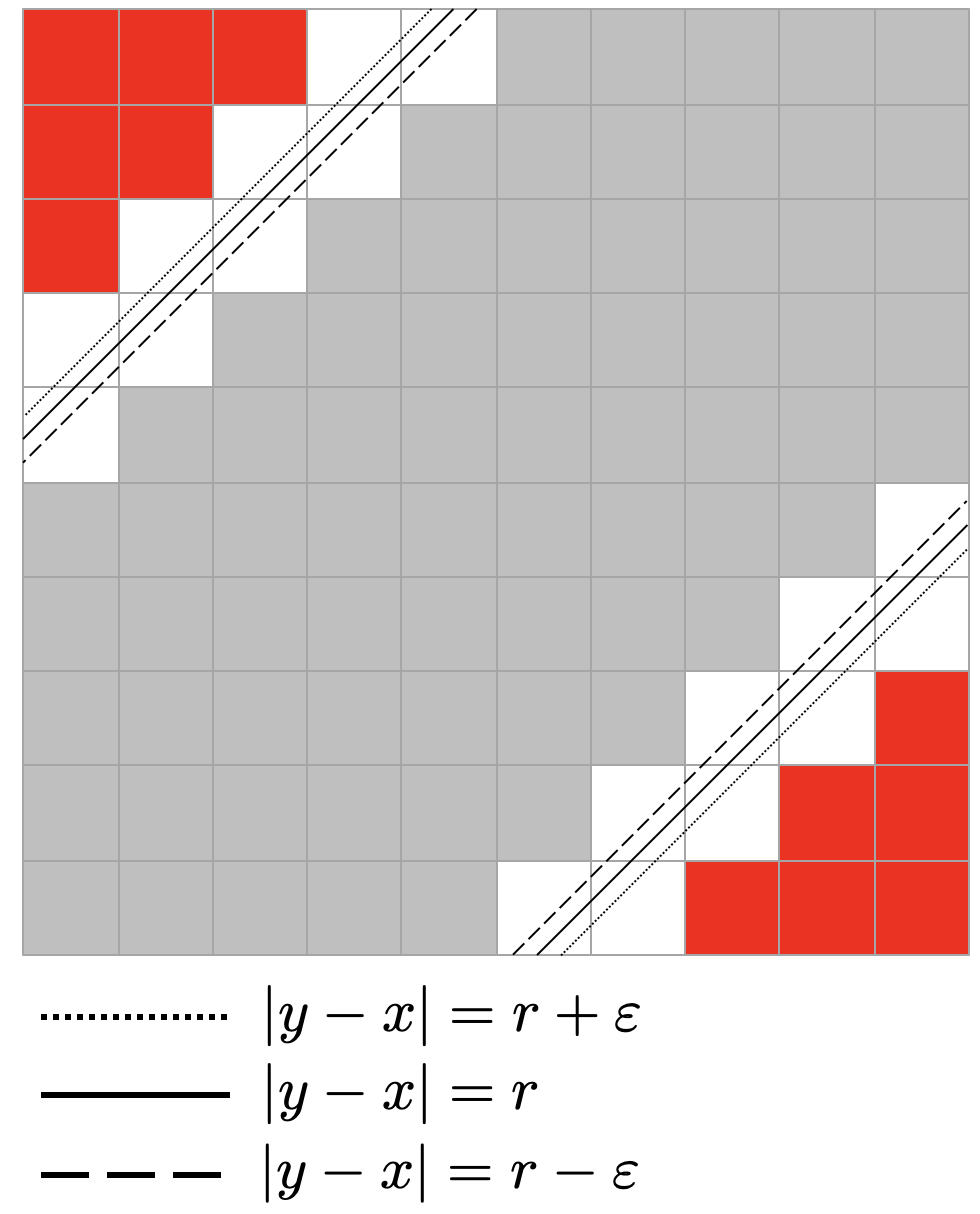}
		\caption{\label{fig:RGG1} 
		The  box $[-1,1]^2$ divided into the cells $L_i\times L_j$.
		The gray cells (middle block) correspond to the set $\mathcal{ G}_1^\varepsilon$ while the red cells (corners) correspond to $\mathcal{ G}_2^\varepsilon$.}
	\end{figure}
	
	\begin{lemma}\label{lem:goodset}
		For every $\varepsilon>0$, almost surely, there exists (a random) $N_\varepsilon\in\mathbb N$  such that, for all $n\ge N_\varepsilon$,  the following two statements are true:
		\begin{enumerate}
			\item If $(i,j)\in \G_1^\varepsilon$ then $h_r(X^{(i)},X^{(j)})=1$.
			\item If $(i,j)\in\G_2^\varepsilon$ then $h_r(X^{(i)},X^{(j)})=0$. 
		\end{enumerate}
	\end{lemma}
	
	\begin{proof}
		Fix $\varepsilon>0$. By Lemma \ref{lem:concentration}, there exists (almost surely) $N\in\mathbb N$ such that, for all $n\geq N$, 
		\begin{align*}
		\sup_{1\le i \le n}\l|X^{(i)} - \E[X^{(i)}]\r|\le n^{-1/3}.
		\end{align*}
		Furthermore, by increasing the value of $N$, we can almost surely find $N_\varepsilon\in\mathbb N$ such that $n^{-1/3}<\varepsilon/2$ for all $n\geq N_\varepsilon$. 
		For all $1\le i,j\le n$, we then have
		\begin{align*}
		|X^{(i)}-X^{(j)}|
		&\le |X^{(i)}-\E[X^{(i)}]|+|\E[X^{(i)}]-\E[X^{(j)}]|+|\E[X^{(j)}]-X^{(j)}|\\
		& \leq |\E[X^{(i)}]-\E[X^{(j)}]|+\varepsilon.
		\end{align*}
		Finally, note that, by \eqref{eq:The_expectation_of_X^(i)}, $\mathbb E[X^{(i)}]\in L_i$ for all $1\leq i\leq N$, which implies $\l|\E[X^{(i)}]-\E[X^{(j)}]\r|<r-\varepsilon$ for all $(i,j)\in \G_1^\varepsilon$, and hence
		\begin{align*}
		|X^{(i)}-X^{(j)}| < r-\varepsilon+\varepsilon=r,
		\end{align*}
		as required. 

		Similarly, we can show that $h_r(X^{(i)}, X^{(j)}) =0$, i.e., $|X^{(i)}-X^{(j)}|>r$ for all $(i,j)\in \G_2^\varepsilon$,  completing the proof.
	\end{proof}
	
	\begin{lemma}\label{lem:approx}
		For $0<r<2$, let $H_r$ and $H_{n,r}$  be as defined above in \eqref{eqn:H} and \eqref{eqn:H_n}. Then, almost surely,
		\begin{align*}
		\limninf	\sup_{x\in B^1}\l| H_{n,r}(x)-H_r(x)\r| = 0.
		\end{align*}
	\end{lemma}
	
	\begin{proof}
		Using \eqref{eqn:H} and \eqref{eqn:H_n}, if $x\in L_i$, then
		\[
		\begin{split}
		\l|H_{n,r}(x)-H_r(x)\r| &= \l| \frac{1}{n}\sum_{j=1}^n h_r(X^{(i)}, X^{(j)}) - \int_{B^1} h_r(x,u)d\nu_1(u)\r| \\
		&\le \sum_{j=1}^n\l|\frac{1}{n}h_r(X^{(i)},X^{(j)})- \int_{L_j}h_r(x,u)d\nu_1(u)\r|.
		\end{split}
		\]
		Fix $\varepsilon >0$. By Lemma \ref{lem:goodset}, there exists $N_\varepsilon\in\mathbb N$ such that for all $n\geq N_\varepsilon$, if $(i,j)\in \G_1^\varepsilon$, then $h_r(X^{(i)}, X^{(j)})=1$ and if $(i,j)\in \G_2^\varepsilon$, then $h_r(X^{(i)},X^{(j)})=0$. In addition, since $x\in L_i$ if $(i,j)\in\G_1^\varepsilon$, then for all $u\in L_j$ we have $h_r(x,u) = 1$ and for all $(i,j)\in\G_2^\varepsilon$, we have $h_r(x,u) = 0$. Consequently, if $(i,j)\in \G_1^\varepsilon\cup \G_2^\varepsilon$, then 
		\begin{align*}
		\l|\frac{1}{n}h_r(X^{(i)},X^{(j)})- \int_{L_j}h_r(x,u)d\nu_1(u)\r|=0,
		\end{align*}
		where we used the fact that $\nu_1(L_i)=n^{-1}$ for all $1\leq i\leq n$. 

		Let $\mathcal B^{(i)}=\{j : (i,j)\not\in \G_1^\varepsilon\cup \G_2^\varepsilon \}$. By the previous argument
		\begin{align*}
		\l|H_{n,r}(x)-H_r(x)\r|\leq \sum_{j\in \mathcal B^{(i)}}\l|\frac{1}{n}h_r(X^{(i)},X^{(j)})- \int_{L_j}h_r(x,u)d\nu_1(u)\r| \le \frac{|	\mathcal B^{(i)}|}{n}.
		\end{align*}
		Note that for a fixed $i$, if $j\in \mathcal B^{(i)}$ then there exists $y_0\in L_j$ such that $r-\varepsilon\leq |x-y_0|\leq r+\varepsilon$, and hence, for all $y\in L_j$,
		\[
			r-\varepsilon-\frac{2}{n}<|x-y_0|-|y-y_0|\leq |x-y|\leq |x-y_0|+|y-y_0|\leq r+\varepsilon+\frac{2}{n}\,.
		\]
		In particular, for every $n\geq N_\varepsilon$, and every $x\in B^1$ such that $x\in L_i$, if $j\in \mathcal B^{(i)}$, then $r-2\varepsilon\leq |x-y|\leq r+2\varepsilon$, for all $y\in L_j$. Let 
		$$
		\Omega_{2\varepsilon,r}=\{(x,y)\in [-1,1]^2 \; :\; r-2\varepsilon\leq |x-y|\leq r+2\varepsilon\}.
		$$ 
		(See Figure \ref{fig:RGG2}.)  Since the length of each set $L_j$ is $2/n$, we obtain the bound
		\begin{align*}
		|\mathcal B^{(i)}|\le \frac{|\Omega_{2\varepsilon,r}|}{2/n}\le \frac{8(2-r)\varepsilon }{2/n} = 4(2-r)\varepsilon n.
		\end{align*}
		\begin{figure}[h]
			\includegraphics[scale=0.3]{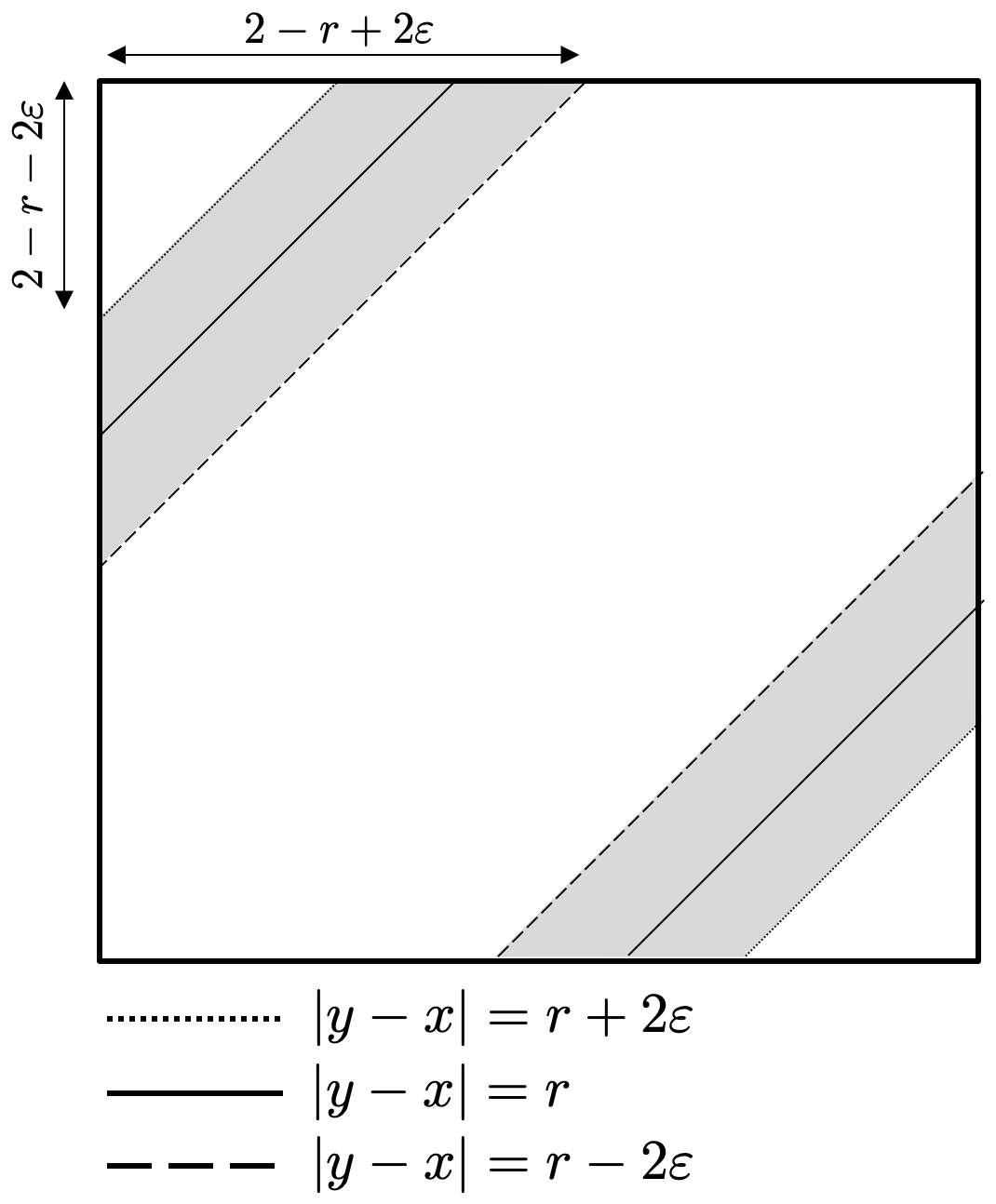}
			\caption{\label{fig:RGG2} The total area of the shaded regions ($\Omega_{2\varepsilon,r}$) is $8\varepsilon(2-r)$.}
		\end{figure}

		Therefore, we have
		\[
		\sup_{x\in B^1}|H_{n,r}(x)-H_r(x)| \le 4(2-r)\varepsilon.  
		\]
		Since $\varepsilon$ was arbitrary, we conclude that, almost surely, the limit is zero.
	\end{proof}
	
	Finally, we are ready to prove the main result of this section.
	\begin{lemma}\label{lem:kernelconvergence}
		For every $0<r<2$, almost surely, $K_{n,r} \to K_r^{1}$ in the cut-norm.
	\end{lemma}
	
	\begin{proof}
		Since the cut-norm is bounded by the $L^1$ norm, it is enough to show that, almost surely, $K_{n,r} \to K_r^1$ in $L^1$, i.e.~that 
		\begin{align*}
		\int_{B^1\times B^1} |K_{n,r}(x,y)-K_r^1(x,y)|d\nu_1(x)d\nu_1(y)\to 0\,,
		\end{align*}
		as $n\to\infty$. 		
		
		For $x\in B^1$, let $1\leq {\mathrm i}(x)\leq n$ to be the unique index such that $x\in L_{{\mathrm i}(x)}$. Fix $\varepsilon>0$, recall the notation in \eqref{eqn:G} and define $\mathcal B^\varepsilon = \l\{(i,j)\in [n]^2 :  (i,j) \not \in \G_1^\varepsilon\cup\G_2^\varepsilon\r\}$. Then, almost surely, there exists $N_\varepsilon\in\N$ such that, for all $n\geq N_\varepsilon$,
		\begin{equation}\label{K_n,r-K_r^1_estimation}
			|K_{n,r}(x,y)-K_r^1(x,y)|\leq \begin{cases}
			 \frac{64\varepsilon}{3r^3} & ({\mathrm i}(x),{\mathrm i}(y))\in \G_1^\varepsilon\\
			 0 & ({\mathrm i}(x),{\mathrm i}(y))\in \G_2^\varepsilon\\
			 \frac{6}{r}& ({\mathrm i}(x),{\mathrm i}(y))\in \mathcal B^\varepsilon\\			 
			\end{cases}\,.
		\end{equation}				
In fact, if $({\mathrm i}(x),{\mathrm i}(y))	\in \G_1^\varepsilon$, then $h_r(x,y)=1$, and from Lemma \ref{lem:goodset} $h_r(X^{(i)}, X^{(j)})=1$. Therefore,  
		\[
		|K_{n,r}(x,y)-K_r^1(x,y)| = \l|\frac{1}{\sqrt{H_{n,r}(x)H_{n,r}(y)}} - \frac{1}{\sqrt{H_r(x)H_r(y)}} \r|=\frac{N_r(x,y)}{D_r^{(1)}(x,y)D_r^{(2)}(x,y)},
		\]
		where 
		\begin{align*}
		N_r(x,y)&:=\l|H_r(x)H_r(y) - H_{n,r}(x) H_{n,r}(y)\r|\\
		D_r^{(1)}(x,y)&:=\l| \sqrt{H_r(x)H_r(y)}+\sqrt{H_{n,r}(x)H_{n,r}(y)}\r|\\
		D_r^{(2)}(x,y)&:=\sqrt{H_{n,r}(x)H_{n,r}(y)H_r(x)H_r(y)},
		\end{align*}
		and 
		\[
		H_{n,r}(x) = \frac{1}{n}\sum_{p=1}^n h_r(X^{({\mathrm i}(x))}, X^{(p)})\qquad\text{and}\qquad H_{n,r}(y) = \frac{1}{n}\sum_{q=1}^n h_r(X^{({\mathrm i}(y))}, X^{(q)}).
		\]
		Recall from \eqref{eqn:H_r_01} and \eqref{eqn:H_r_12} that $\frac{r}{2}\le H_r(x)\le 1$ for $x\in B^1$. Hence, by Lemma \ref{lem:approx}, we have  that, almost surely, for all sufficiently large $n$,
		\begin{equation}\label{eq:bafjklbalfkbvja}
				\frac{r}{4}\leq H_r(x)-\frac{r}{4}\leq H_{n,r}(x),
		\end{equation}
		and 
		\[
			\sup_{x\in B^1}\l| H_{n,r}(x)-H_r(x)\r| \leq \varepsilon\,.
		\]
		Observe  that $H_{n,r}(x)\le 1$. Hence we conclude that 
		\begin{align*}
			D_r^{(1)}(x,y)\ge \frac{3r}{4}\qquad\text{and}\qquad D_r^{(2)}(x,y)\ge \frac{r^2}{8}\,,
		\end{align*}
		and		
		\[
		N_r(x,y)\le |H_r(x)||H_r(y)-H_{n,r}(y)|+|H_{n,r}(y)||H_r(x)-H_{n,r}(x)|\leq 2\varepsilon\,.
		\]
		Combining all of the above we conclude that almost surely, for all sufficiently large $n$
		\[
			|K_{n,r}(x,y)-K_r^1(x,y)|\leq \frac{2\varepsilon}{\frac{3r}{4}\cdot\frac{r^2}{8}}\leq \frac{64\varepsilon}{3r^3}\,.
		\]
		
		Turning to the second case in \eqref{K_n,r-K_r^1_estimation}, note that if $({\mathrm i}(x),{\mathrm i}(y))\in \G_2^\varepsilon$, then $h_r(x,y)=0$, and from Lemma \ref{lem:goodset} also $h_r(X^{(i)}, X^{(j)})=0$. Therefore, $|K_{n,r}(x,y)-K_r^1(x,y)|=0$. 
		
		Finally, if $({\mathrm i}(x),{\mathrm i}(y))\in \mathcal B^\varepsilon$, then from \eqref{eq:bafjklbalfkbvja}, the fact that $|h_r(x,y)|\leq 1$ for all $x,y\in B^1$ and the definitions of $K_{n,r}$ and $K_r^1$, it follows that 
		\[
			|K_{n,r}(x,y)-K_r^1(x,y)|\leq |K_{n,r}(x,y)|+|K_r^1(x,y)| \leq \frac{4}{r}+\frac{2}{r} = \frac{6}{r}\,.
		\]		
		
		Combining \eqref{K_n,r-K_r^1_estimation} with the fact that $(\nu_1\times\nu_1)(L_i\times L_j) = n^{-2}$, we conclude that 
		\begin{equation}\label{eq:blablabla}
		\begin{aligned}
			\int_{B^1\times B^1} |K_{n,r}(x,y)-K_r^1(x,y)|d\nu_1(x)d\nu_1(y)\leq \frac{64\varepsilon}{3r^3}\frac{|\G_1^\varepsilon|}{n^2}+\frac{6}{r}\frac{|\mathcal B^\varepsilon|}{n^2}\leq \frac{64\varepsilon}{3r^3}+\frac{6|\mathcal B^\varepsilon|}{rn^2}\,,
		\end{aligned}
		\end{equation}
		where in the last inequality we used the fact that $|\mathcal G_1^\varepsilon|\leq n^2$. 
		
		Next, we bound the size of $\mathcal B^\varepsilon$. Note that, for every $x,y\in B^1$, if $({\mathrm i}(x),{\mathrm i}(y))\in\mathcal{B}^\varepsilon$, then $r-2\varepsilon<|x-y|<r+2\varepsilon$, 
		and so 
		\[
			(\nu_1\times\nu_1)\big(\{(x,y) : ({\mathrm i}(x),{\mathrm i}(y))\in \mathcal{B}^\varepsilon\}\big)\le \nu_1\times\nu_1 (\Omega_{2\varepsilon,r})= 2(2-r)\varepsilon\,.
		\]
		Since the sets $(L_i\times L_j)_{i,j=1}^n$ are disjoint, cover $B^1\times B^1$ and $(\nu_1\times \nu_1)(L_i\times L_j)=n^{-2}$ for all $1\leq i,j\leq n$, it follows that
		\[
			|\mathcal B^\varepsilon|\leq \frac{(\nu_1\times\nu_1)\big(\{(x,y) : ({\mathrm i}(x),{\mathrm i}(y))\in \mathcal{B}^\varepsilon\}\big)}{\frac{1}{n^2}}\leq 2(2-r)\varepsilon n^2 . 
		\]
		Combining the last bound together with \eqref{eq:blablabla}, we conclude that for all sufficiently large $n$
		\[
			\int_{B^1\times B^1} |K_{n,r}(x,y)-K_r^1(x,y)|d\nu_1(x)d\nu_1(y)\leq \frac{64\varepsilon}{3r^3}+\frac{12(2-r)\varepsilon }{r}\,.
		\]	
		Since $\varepsilon>0$ was arbitrary, the result follows. 
	\end{proof}

	
	\section{ construction of the kernel for $K_{n,r}$  $d= 2$}\label{sec:kernel_d2}
	Our next goal is to generalize the results from the previous section to arbitrary dimension $d\geq 2$. That is, to construct a sequence of  kernels $K_{n,r}:B^d\times B^d\to \R$ that possess the same spectra as the symmetrically normalized adjacency operators and converge in the cut norm to $K_r^d$.   
	
	 Recall that the kernel                $K_{n,r}$, defined in \eqref{eqn:K_n_formula} for $d=1$, uses an ordering of the points based on their (single) coordinate value. 
	The main challenge now is how to choose a similar ordering on the points, when $d\ge 2$. Throughout this section, instead of considering the kernels $K_{n,r}$ for arbitrary choice of $n$, we only examine the case where $n=m^d$ for some $m\in \N$. This will help us devising the required ordering on the points. Later on, in Section \ref{sec:thm1}, we will show how to extend the results from $n=m^d$ to any $n$. Finally, since the case $d=2$ is considerably simpler than the general case, we start by providing all the details for $d=2$ in this section. The general construction, which is done similarly, is outlined in Section \ref{sec:generald}. 
	
	\subsection{Kernel definition and spectrum.}

As mentioned above, the challenging part here is to define a useful ordering on the $d$-dimensional points $X_1\ldots, X_n$.	Assume that $n=m^2$ for some $m\in\N$, and let $X_{i}=(X_{i,1}, X_{i,2})$,  $i=1,\ldots, n$, be i.i.d.~uniformly distributed random variables in $B^2=[-1,1]^2$. We  order the points $X_1,\ldots, X_n$ and rename them in the following way: 
	
	\vspace{10pt}
	\noindent {\bf Step I :} We order $X_1,\ldots, X_n$ according to the order statistics of the first coordinates $X_{1,1}, \ldots, X_{n,1}$, and denote the resulting points by $X^{(1)},\ldots,X^{(n)}$. In other words, if $X^{(i)}=(X_1^{(i)},X_2^{(i)})$ for $i=1,\ldots, n$, then,  for $ i< j$,
	\[
	X_1^{(i)}\le X_1^{(j)}.
	\]
	
	\vspace{.2cm}
	\noindent {\bf Step II :} We take the first $m$ variables $X^{(1)}, \ldots, X^{(m)}$, and re-order them according to the order statistics of the second coordinate $X_{2}^{(1)}, \ldots,X_{2}^{(m)}$. The resulting ordering is denoted by $X^{(1,1)},\ldots, X^{(1, m)}$, so that, if  $X^{(1,i)}=(X_1^{(1,i)},X_2^{(1,i)})$ for $i=1,\ldots, m$, then, for $1\le i < j\le {m}$,
	\begin{align*}
	X_2^{(1,i)}\le X_2^{(1,j)}.
	\end{align*}
	
	\noindent {\bf Step III :} We order each of the $m$-tuples in a similar fashion. For $p =2,\ldots,{m}$, take $X^{((p-1) m+1)},$ $ \ldots, X^{(pm)}$, and sort them according to the order statistics of the second coordinate $X_2^{((p-1) m+1)},$ 
	$\ldots,X_2^{(p m)}$. The resulting ordered random variables  are denoted by $X^{(p,1)}, \ldots, X^{(p, m)}$, so that  $X^{(p,i)}=(X_1^{(p,i)},X_2^{(p,i)})$, $i=1,\ldots, m$, and, for every $1\le i < j \le {m}$,
	\begin{align*}
	X_2^{(p,i)}\le X_2^{(p,j)}.
	\end{align*}
	
	The result is a collection of indexed variables $(X^{(i,j)})_{i,j=1}^m$ with the property such that $X_1^{i,k}\leq X_1^{j,l}$ for all $1\leq i\leq j\leq m$ and $1\leq k,l\leq m$, and $X_2^{k,i}\leq X_2^{k,j}$ for all $q\leq i,j\leq m$ and all $1\leq k\leq m$. This new ordering of the points will play a crucial role for the construction of  $K_n^{2}$.
	

	
			Recall that in order  to define the kernel $K_{n,r}$ for $d=1$ we divided $B^1$ into the intervals $L_{i}$, (cf.\  \eqref{eqn:L_i}) and  that the bulk of the convergence proof relied on the fact that $X^{(i)} \in L_i$ for all $i$ (almost surely for large enough $n$). For $d=2$, we use a similar construction, where we divide $B^2$ into boxes $L_{p,q}^n$ and show that almost surely, for $n$ large enough, we have $X^{(p,q)}\in L_{p,q}^n$ for all $p,q$.
	More concretely, recall the definition of $L_i^n$ in  \eqref{eqn:L_i}, and,  for every $1\leq p,q\leq m$, define $L_{p,q}^n=L_{p,q}$ by 
	\begin{equation}\label{eqn:L_i_d2}
	L_{p,q}^n =L_p^m\times L_q^m \subset B^2.
	\end{equation}
	Note that  $B^2=\bigsqcup_{p,q=1}^{m} L_{p,q}^n$, and $\nu_2(L_{p,q}^n) =m^{-2}=n^{-1}$.

	The kernels defined in this subsection, are similar to the ones from Section \ref{sec:d1}, where instead of $X^{(i)}$ and $L^n_i$ we use $X^{(p,q)}$ and $L_{p,q}^n$.
	
	Recall that $n=m^2$, and define a sequence of random functions  $h_{n,r}\; :\;B^2\times B^2\to \R$ by 
	\begin{align*}
	h_{n,r}(x,y):=\sum_{p,q,p',q'=1}^{m}h_r(X^{(p,q)}, X^{(p',q')})\one_{L_{p,q}^m}(x)\one_{L_{p',q'}^m}(y).
	\end{align*}
	Furthermore, for $x\in B^2$, denote
	\begin{equation}\label{eqn:H_n_d2}
	H_{n,r}(x,y)=\int_{B^2} h_{n,r}(x,y)d\nu_2(y)=\frac{1}{n}\sum_{p,q=1}^m\one_{L_{p,q}}(x)\sum_{p',q'=1}^m h_r(X^{(p,q)}, X^{p',q')}),
	\end{equation}
	where we used the fact that  $\nu_2(L_{p',q'}^n) = n^{-1}$.
	
	Next, define a sequence of random kernels $K_{n,r}\; : \; B^2\times B^2\to \R$ by
	\begin{align*}
	K_{n,r}(x,y)&:=\frac{h_{n,r}(x,y)}{\sqrt {H_{n,r}(x) H_{n,r}(y)}}\,,
	\end{align*}
	or, equivalently, 
	\begin{equation}\label{eqn:K_n_formula_2}
	K_{n,r}(x,y)=\sum_{p,q,p',q'=1}^{m}\frac{nh_r(X^{(p,q)}, X^{(p',q')})\one_{L_{p,q}}(x)\one_{L_{p',q'}}(y)}{\sqrt{\sum_{a_1,a_2,a_3,a_4=1}^{m}h_r(X^{(p,q)}, X^{(a_1,a_2)})h(X^{(p',q')}, X^{(a_3,a_4)})}}.
	\end{equation}
	Finally, let $\mathcal K_{n,r}\; : \; L^2(B^2, \nu_2) \to L^2(B^2, \nu_2)$  be the Hilbert-Schmidt kernel operator corresponding to  the kernel $K_{n,r}$,  i.e.
	\begin{align}\label{eqn:kerneld2}
	\mathcal K_{n,r}f(x)=\int_{B^2} K_{n,r}(x,y)f(y)d\nu_2(y).
	\end{align}
	
	The following is the 2-dimensional analogue of Lemma \ref{lem:equaleigenvalues}.
	\begin{lemma}\label{lem:eigenvaluessame2}
		Suppose that $d=2$ and $0<r<2$, and let  $\mathcal K_{n,r}$ be as defined in \eqref{eqn:kerneld2}.  Then $\spec(\K_{n,r}) = \spec(W_{n,r})$, where $W_{n,r}$ is as defined in \eqref{eqn:W_n}.
	\end{lemma}
	
	\begin{proof}
		The proof is very similar to that of Lemma \ref{lem:equaleigenvalues}, and so we only highlight the differences.
		
		Let $\widetilde A_n$ and $\widetilde W_n$ be the matrices describing the graph generated by the sorted points $X^{(p,q)}$, where we use the lexicographic ordering on the pairs $(p,q)$ as described above. Using a slight abuse of notation, we use quadruplets $(p,q,p',q')$, as entry indices for the matrices $\widetilde A_n, \widetilde W_n$. Since we only changed the order of the original vertices $X_1,\ldots, X_n$, we have $\spec(\widetilde W_n) = \spec(W_n)$.
		
		For $1\leq p,q,p',q'\leq m$, denote 
		\[
			a_{(p,q),(p',q')} = (\widetilde A_n)_{(p,q),(p',q')} = h_r(X^{(p,q)},X^{(p',q')}),
		\]
		and
		\[
			d_{(p,q)} = 	\sum_{p',q'=1}^m a_{(p,q),(p',q')}.
		\]
		Then the $(p,q,p',q')$-th entry of $\widetilde W_n$  can be written as
		\begin{align*}
		w_{(p,q),(p',q')} = \frac{a_{(p,q),(p',q')}}{\sqrt{d_{(p,q)}d_{(p',q')}}}.
		\end{align*}
		Using this notation with \eqref{eqn:K_n_formula_2}, we can also write
		\[
			K_{n,r}(x,y)= n\sum_{p,q,p',q'=1}^{m}w_{(p,q),(p',q')}\one_{L_{p,q}}(x)\one_{L_{p',q'}}(y).
		\]
		The rest of the proof is identical to that of Lemma \ref{lem:equaleigenvalues}.
		
	\end{proof}
	
	\subsection{Concentration statements}
	Similarly to the case $d=1$, we want to show that $X^{(p,q)}\in L_{p,q}^n$.
		\begin{lemma}\label{lem:concentration2}
		Let $(X^{p,q})_{p,q=1}^m$  be the ordering defined above. Then, almost surely, there exists $N\in\N$ such that, for all $n\ge N$, 
		\begin{align*}
		\sup_{1\le p,q \le m}\big\|X^{(p,q)}-\E[X^{(p,q)}]\big\|_\infty\le {n^{-1/6}},
		\end{align*}
		where
				\begin{align*}
		\E[X^{(p,q)}]\in L^n_{p,q}.
		\end{align*}
	\end{lemma}
	
	\begin{proof}
	
	The bound in this lemma can be obtained using \cite[equation (1.1)]{shor}. See \cite[Section 4]{shor89} for the proof, which in fact gives a better bound. For the sake completeness we provide an alternative proof of for this bound using order statistics arguments in Appendix \ref{app1}. 
The remainder of the proof is dedicated to show that, indeed, $\E[X^{(p,q)}]\in L^n_{p,q}$.

		Denote by $\underline X_1=(X_1^{(1)},\ldots, X_1^{(n)})$ the vector of first coordinates of all points, and suppose that $\underline X_1$ is given.
		In this case, by Step II and Step III, for all $p = 1,\ldots, m$, the values of $X_1^{(p,1)}, \ldots, X_1^{(p,m)}$ are    {the same as those in }  the sequence $X_1^{((p-1)m+1)}, \ldots, X_1^{(pm)}$, {under}  a random  permutation (since they are  ordered according to the values of the second coordinates $X_2^{((p-1)m+1)},$ $\ldots, X_2^{(pm)}$, which are i.i.d.\ and independent of the first coordinate). Therefore, 
		\begin{equation}\label{eqn:E_X_pq_0}
	  		\E[X^{(p,q)}_1 \given \underline X_1] = \frac{1}{m} \sum_{s=1}^m X_1^{((p-1)m + s)}. 
		\end{equation}
		Next, from Step I we have that $X_1^{(1)}, \ldots, X_{1}^{(n)}$ are the order statistics of $n=m^2$ i.i.d.~uniformly distributed random variables in $[-1,1]$. Thus, using \eqref{eq:The_expectation_of_X^(i)},
		\begin{equation}\label{eqn:E_X_pq}
		\begin{split}
		\E[X^{(p,q)}_1] &=  \frac{1}{m} \sum_{s=1}^{ m}\frac{2((p-1) m+s)}{ n+1}-1\\
		&=-1+\frac{2pm -m +1}{n+1},
		\end{split}
		\end{equation}
		and  it follows that $\E[X^{(p,q)}_1] \in L_{p,m}$.
		
		Next, fix $p$, and notice that  given $\underline X_1$ we have that $X^{(p,1)}_2\le \cdots \le X^{(p, m)}_2 $  are the order statistics of 
		$m$ i.i.d.~random variables, uniformly distributed in $[-1,1]$. Therefore, using \eqref{eq:The_expectation_of_X^(i)} again, gives 
		\[
		\E[X_2^{(p,q)} \given \underline X_1] = -1 + \frac{2q}{m+1},
		\]
		which implies that $\E[X_2^{(p,q)}] = -1 + \frac{2q}{m+1} \in L_{q,m}$.
		To conclude, we showed that for all $1\le p,q \le m$, 
		\begin{align*}
		(\E[X^{(p,q)}_1], \E[	X^{(p,q)}_2])\in  L_{p}^m\times L_{q}^m=L_{p,q}^n,
		\end{align*}
	as required.
	\end{proof}
	

	\subsection{The convergence of $K_{n,r}$.} 
	
	In this section we show that $K_{n,r}$ converges to $K^{2}_r$ in the cut-norm, almost surely, as  $n\to \infty$. The proofs leading to this statement follow  steps similar to  those in Section \ref{sec:K_converge}, and so we only highlight the main differences.
	
	Fix $0<r<2$ and $\varepsilon>0$. Similarly to \eqref{eqn:G}, we start by defining the sets 
	\begin{equation}\label{eqn:G_12}
	\begin{split}
	{ \mathcal G}_1^{\varepsilon}&:=\{((p,q),(p',q')) ~:~ \|x-y\|< r-\varepsilon, \mbox{ for all }(x,y)\in L_{p,q}\times L_{p',q'} \},\\
	{\mathcal  G}_2^\varepsilon & :=\{((p,q),(p',q')) ~:~ \|x-y\|> r+\varepsilon, \mbox{ for all }(x,y)\in L_{p,q}\times L_{p',q'} \}.
	\end{split}
	\end{equation}
	
	We start by proving the analogue of Lemma \ref{lem:goodset}.
	\begin{lemma}\label{lem:goodset2}
		Almost surely, there exists (random) $N_\varepsilon>0$ such that, for all $n\ge N_\varepsilon$,  the following two statements are true:
		\begin{enumerate}
			\item If $((p,q),(p',q'))\in {\mathcal G}_1^\varepsilon$, then $h_r(X^{(p,q)},X^{(p',q')})=1$.
			\item If $((p,q),(p',q'))\in {\mathcal G}_2^\varepsilon$, then $h_r(X^{(p,q)},X^{(p',q')})=0$.
		\end{enumerate}
	\end{lemma}
	
	\begin{proof}
		For $1\leq p,q,p',q'\leq m$,
		\begin{align*}
		&\big\|X^{(p,q)}-X^{(p',q')}\big\|_\infty \\
		&\qquad\qquad \le \big\|X^{(p,q)}-\E[X^{(p,q)}]\big\|_\infty+\big\|\E[X^{(p,q)}]-\E[X^{(p',q')}]\|_\infty+\big\|\E[X^{(p',q')}]-X^{(p',q')}\big\|_\infty.
		\end{align*}
		Lemma \ref{lem:concentration2} implies that $\E[X^{(p,q)}]\in L_{p,q}$ and $\E[X^{(p',q')}]\in L_{p',q'}$. Thus, if  $((p,q),(p',q'))\in { {\mathcal G}}_1^{\varepsilon}$, then
		\begin{align*}
		\big\|\E[X^{(p,q)}]-\E[X^{(p',q')}]\big\|_\infty< r-\varepsilon.
		\end{align*}
		Lemma \ref{lem:concentration2} also implies that a.s.~there exists $N_1$ such that , for  $n\ge N_1$,  and for all $1\leq p,q\leq m$,
		\begin{align*}
		\big\|X^{(p,q)}-\E[X^{(p,q)}]\big\|_\infty\le n^{-1/6}.
		\end{align*}
		Choosing $N_1$ such that $N_1^{-{1/6}}<\varepsilon/2$ and combining the last two estimates, we have that, for $n\ge \max\{N_0,N_1\}$, 
		\begin{align*}
		\sup_{((p,q),(p',q')) \in {\mathcal G}_1^{\varepsilon}}\|X^{(p,q)}-X^{(p',q')}\|_\infty<r,
		\end{align*}
		implying that $h_r(X^{(p,q)},X^{(p',q')})=1$ for all $((p,q),(p',q'))\in \mathcal G_1^\varepsilon$.

A similar computation shows that $h_r(X^{(p,q)},X^{(p',q')})=0$ for all $((p,q),(p',q'))\in \mathcal G_2^\varepsilon$, thus completing the proof with $N_\varepsilon=\max\{N_0,N_1\}$. 
		%
%
	\end{proof}
	
	Next, we prove a result analogous  to Lemma \ref{lem:approx}.
	
	\begin{lemma}\label{lem:approx2}
		Let $H_r$ and $H_{n,r}$ be as defined in \eqref{eqn:H} and \eqref{eqn:H_n_d2} respectively. Then, almost surely
		\begin{align*}
		\limninf	\sup_{x\in B^2}\l| H_{n,r}(x)-H_r(x)\r| = 0.
		\end{align*}
	\end{lemma}
	
	\begin{proof}
		The proof here is identical to that of Lemma \ref{lem:approx}.
	\end{proof}
	
	%
	
	Finally, we prove the main result of this section.
	\begin{lemma}\label{lem:convergencekernel2}
		Let $K_{n,r}$  be as defined above. Then $K_{n,r}\to K^{2}_r$, with respect to the cut-norm, almost surely,  as $n\to\infty$.
	\end{lemma}
	
	\begin{proof}
		The proof is similar to that of Lemma \ref{lem:kernelconvergence}, and again we  highlight only the necessary changes. Fix $0<r<2$ and $\varepsilon>0$, and define 
		\[
		\mathcal B^\varepsilon = \l\{((p,q),(p',q'))\in [m]^4 :  ((p,q),(p',q')) \not \in \G_1^\varepsilon \cup \G_2^\varepsilon\r\}.
		\]
		
		For $x\in B^2$, define ${\mathrm p}(x),{\mathrm q}(x)$ to be the unique integers in $[m]$ such that $x\in L_{{\mathrm p}(x),{\mathrm q}(x)}^m$. A similar argument to the one in the one-dimensional case shows that  
			\begin{equation}\label{K_n,r-K_r^1_estimation=2}
			|K_{n,r}(x,y)-K_r^2(x,y)|\leq \begin{cases}
			 \frac{512\varepsilon}{3r^6} & (({\mathrm p}(x),{\mathrm q}(x)),({\mathrm p}(y),{\mathrm q}(y)))\in \G_1^\varepsilon\\
			 0 & (({\mathrm p}(x),{\mathrm q}(x)),({\mathrm p}(y),{\mathrm q}(y)))\in \G_2^\varepsilon\\
			 \frac{12}{r^2}& (({\mathrm p}(x),{\mathrm q}(x)),({\mathrm p}(y),{\mathrm q}(y)))\in \mathcal B^\varepsilon\\			 
			\end{cases}\,,
		\end{equation}
and therefore 
		\begin{equation}\label{eq:eavalvsvs}
		\begin{aligned}
			\int_{B^2\times B^2}|K_{n,r}(x,y)-K_r^2(x,y)|d\nu_2(x)d\nu_2(y) &\leq \frac{512\varepsilon}{3r^6}\cdot \frac{|\G_1^\varepsilon|}{n^2} + \frac{12}{r^2}\cdot \frac{|\mathcal{B}^\varepsilon|}{n^2}\\
			&\leq \frac{512\varepsilon}{3r^6}+ \frac{12}{r^2}\cdot \frac{|\mathcal{B}^\varepsilon|}{n^2}.
		\end{aligned}
		\end{equation}

		Thus, it remains to bound the size of $\mathcal B^\varepsilon$. Note that if $(x,y)\in B^2\times B^2$ is a pair of points such that $(({\mathrm p}(x),{\mathrm q}(x)),({\mathrm p}(y),{\mathrm q}(y)))\in \mathcal B^\varepsilon$, then $r-2\varepsilon\leq \|x-y\|_\infty \leq r+2\varepsilon$, and therefore, either $r-2\varepsilon<|x_1-y_1|<r+2\varepsilon$ or $r-2\varepsilon<|x_2-y_2|<r+2\varepsilon$. Hence,
		\[
		\begin{aligned}
			&\nu_2\times\nu_2\big(\{(x,y)\in B^2\times B^2 ~:~ (({\mathrm p}(x),{\mathrm q}(x)),({\mathrm p}(y),{\mathrm q}(y)))\in \mathcal B^\varepsilon\}\big)\\
			&\qquad\qquad \leq  \nu_2\times\nu_2\big(\{(x,y)\in B^2\times B^2 ~:~ r-2\varepsilon<|x_1-y_1|<r+2\varepsilon\big)\\
			&\qquad\qquad\qquad\qquad +  \nu_2\times\nu_2\big(\{(x,y)\in B^2\times B^2 ~:~ r-2\varepsilon<|x_2-y_2|<r+2\varepsilon\big)\\
			&\qquad\qquad \leq  2(2-r)\varepsilon + 2(2-r)\varepsilon = 4(2-r)\varepsilon.
		\end{aligned}
		\]
		Since $(L_{p,q}^m\times L_{p',q'}^m)_{p,q,p',q'=1}^m)$ are disjoint, cover $B^2\times B^2$ and each one satisfies  $\nu_2\times\nu_2(L_{p,q}^m\times L_{p',q'}^m)$ $=n^{-2}$, it follows that 
%
		\begin{align*}
			|\mathcal B^\eps|\le \frac{\nu_2\times\nu_2(\mathcal{B}_\varepsilon)}{n^{-2}} = 4(2-r)\varepsilon n^2.
		\end{align*}
%
Substituing the last bound into \eqref{eq:eavalvsvs} shows that, for all $n\geq N_\varepsilon$,
\[
	\int_{B^2\times B^2}|K_{n,r}(x,y)-K_r^2(x,y)|d\nu_2(x)d\nu_2(y) \leq \frac{512 \varepsilon}{3r^6}+ \frac{48(2-r)\varepsilon}{r^2}\,,
\]
and, since $\varepsilon>0$ was arbitrary, the result follows.
	\end{proof}
	
	\section{ Outline of the construction of $K_{n,r}$ for general $d$.}\label{sec:generald} 

	In this section we show how to construct the kernel $K_{n,r}$ for $d\ge 3$ and for $n = m^d$ for some $m\in \N$. Later, we will show how to prove the results for arbitrary values of $n\in \N$. The construction as well as the proofs are similar to the case $d=2$, just a bit more technically involved. Therefore, in this section we only wish to provide an outline for the general case, without repeating all the details and proofs.
	
	Let $X_1,\ldots, X_n$ be $n$ i.i.d.\ points, uniform in $B^d$, and denote $X_i = (X_{i,1},\ldots, X_{i,d})$. 
	As in the $d=2$ case, the tricky part here is to provide a useful ordering on the vertices. This is done in a sequence of $d$ steps as follows.
	
	\vspace{.2cm}
	\noindent{\bf Step 1:} Order $X_1, \ldots, X_n$ according to the first coordinate, and denote the result by $X^{(1)}\ldots X^{(n)}$. Thus, if  $X^{(i)}=(X_1^{(i)}, \ldots, X_d^{(i)})$ then for all $1\le i<j\le n$,
	\begin{align*}
	X_1^{(i)} \le X_1^{(j)}.
	\end{align*}

	\noindent {\bf Step 2:} 
	Take the variables $X^{(1)},\ldots, X^{(m^{d-1})}$ and order them using the second coordinates $X_2^{(1)},\ldots, X_2^{(m^{d-1})}$.
	Similarly, for all  $i_1 = 1,\ldots, m$, take the $i_1$-th collection of  the $m^{d-1}$ variables \allowbreak $X^{((i_1-1)m^{d-1}+1)},\ldots, X^{(i_1m^{d-1})}$, and order them according to the values in the second coordinate. Denote the result $X^{(i_1, i_2)}$, where $1\le i_1\le m$, and $1\le i_2\le m^{d-1}$.
	
	In the end of this sorting process, from the first step we have that
	for all $1 \le i_1 < i_1' \le m$, and for all $1\le i_2,i_2' \le m^{d-1}$,
	\[
	X_1^{(i_1, i_2)} \le X_1^{(i_1', i_2')}.
	\]
	In addition, if we fix $1\le i_1 \le m$, then, from the second step, for all $1\le i_2 < i_2' \le m^{d-1}$,
	\[
	X_2^{(i_1,i_2)} \le X_2^{(i_1,i_2')}.
	\]
	
	\vspace{.2cm}
	\noindent{\bf Step 3:} Take the variables $X^{(1,1)}, \ldots, X^{(1, m^{d-2})}$, and order them according to the third coordnates $X_3^{(1,1)}, \ldots, X_3^{(1, m^{d-2})}$.
	Similarly, for all $1\le i_1,i_2 \le m$, take the collection of the $m^{d-2}$ variables $X^{((i_1,(i_2-1)m^{d-2}+1)},\ldots, X^{(i_1,i_2m^{d-2})}$ and order them according to the third coordinate. Denote the result $X^{(i_1,i_2,i_3)}$, for $1\le i_1,i_2 \le m$, and $1\le i_3 \le m^{d-2}$.
	
	In the end of this sorting process, from the first step we have that
	for all $1 \le i_1 < i_1' \le m$, for all $1\le i_2,i_2' \le m$, and for all $1\le i_3,i_3' \le m^{d-2}$,
	\[
	X_1^{(i_1, i_2, i_3)} \le X_1^{(i_1', i_2', i_3')}.
	\]
	Next, fixing $1\le i_1 \le m$, then from the second step for all $1\le i_2 < i_2' \le m$, and for all $1 \le i_3,i_3' \le m^{d-2}$, we have
	\[
	X_2^{(i_1,i_2, i_3)} \le X_2^{(i_1,i_2', i_3')}.
	\]
	Finally, fixing $1\le i_1,i_2 \le m$, then from the third step for all $1 \le i_3 < i_3' \le m^{d-2}$ we have
	\[
	X_3^{(i_1,i_2,i_3)} \le X_3^{(i_1,i_2,i_3')}.
	\]

	\vspace{.2cm}
	\noindent{\bf Step k:}
	We keep performing these sorting procedure in a similar way. For the $k$-th step, for every choice of $1\le i_1,\ldots, i_{k-1} \le m$, we take  collections of the $m^{d-k+1}$ variables from the previous step, i.e.~$X^{(i_1,\ldots, i_{k-2},(i_{k-1}-1)m^{d-k+1}+1)}, \ldots, X^{(i_1,\ldots, i_{k-2},i_{k-1}m^{d-k+1})}$, and order them according to the $k$-th coordinate. The result is denoted $X^{(i_1,\ldots, i_k)}$. This will be done for all $k\le d$.
	
	Concluding this procedure, we take the $d$-dimensional variables $X_1,\ldots, X_n$ and order them in a sequence of $d$ steps, coordinate by coordinate, until we reach the sorted sequences $X^{(i_1,\ldots, i_d)}$, where the indices are $1\le i_1,\ldots, i_d \le m$.
	For brevity we will use $\bi = (i_1,\ldots, i_d)$, and $X^{(\bi)} = X^{(i_1,\ldots, i_d)}$. We also define $\one = (1,\ldots, 1)$ and $\bm = (m,\ldots,m)$, and we use `$\le$' to denote  lexicographic order.
	
	Similarly to the $d=2$ case, our next step is to define a useful partition of $B^d$. Suppose that $\bi = (i_1,\ldots, i_d)$ is such that $\one \le \bi \le \bm$. Using the definition of $L_{i,n}$ \eqref{eqn:L_i}, we define
	\begin{equation}\label{eqn:L_i_dd}
	L_{\bi,n} =L_{i_1,m}\times L_{i_2,m} \times\cdots\times L_{i_d,m} \subset B^d.
	\end{equation}
	In this case we have that $B^d = \bigsqcup_{\bi = \one}^{\bm} L_{\bi,n}$, and $\nu_d(L_{\bi,n}) = 1/n$. As before, we  denote $L_{\bi} = L_{\bi,n}$.
	
	Next, we define the kernels, for $x,y\in B^d$, as
	\begin{align}\label{eqn:kerneld}
	K_{n,r}(x,y)=\frac{h_{n,r}(x,y)}{\sqrt{H_{n,r}(x)H_{n,r}(y)}},
	\end{align}
	where 
	\begin{align*}
	h_{n,r}(x,y)=\sum_{\bi,\bj = \one}^{\bm} h_r(X^{(\bi)},X^{(\bj)})\one_{L_{\bi}}(x)\one_{L_{\bj}}(y),
	\end{align*}
	and
	\begin{equation}\label{eqn:H_n_d}
	H_{n,r}(x) = \int_{B^d} h_{n,r}(x,y) d\nu_d(y) = \frac{1}{n} \sum_{\bi} \one_{L_{\bi}}(x) \sum_{\bp} h_r(X^{(\bi)}, X^{(\bp)}).
	\end{equation}
	To prove that $K_{n,r} \to K_r^d$ we will have to prove  lemmas corresponding to those in Sections \ref{sec:kernel_d1} and \ref{sec:kernel_d2}. We will present the lemmas and discuss the needed adjustments for the proofs.
	
	\begin{lemma}\label{lem:spec_d}
		Let  $\mathcal K_{n,r}$ be the Hilbert-Schmidt kernel operator on $L^2(B^d,\nu_d)$ corresponding to   $K_{n,r}$ defined above. Then $\spec(\mathcal K_{n,r})=\spec(W_{n,r})$.
	\end{lemma}
	
	\begin{proof}
		The proof here is identical to that of Lemma \ref{lem:eigenvaluessame2}.
	\end{proof}

	\begin{lemma}\label{lem:unif_bound}
		Let $X^{(\bi)}$  be as defined above. Then, almost surely, there exists $N>0$ such that, for all $n\ge N$, we have,
		\begin{align*}
		\sup_{\one \le \bi \le \bm}\l\|X^{(\bi)}-\E[X^{(\bi)}]\r\|\le \frac{1}{n^{1/3d}},
		\end{align*}
		where 
		\begin{align*}
		\E[X^{(\bi)}]\in L_{\bi}.
		\end{align*}
	\end{lemma}
	
	\begin{proof}
This bound can be proved using \cite[Theorem 1.1]{shor}. But for  completeness we give a proof in Appendix  \ref{app1}. We will explain the steps needed to bound   $\E[X^{(\bi)}]$.

		With, as before, $\bi  = (i_1,\ldots, i_d)$,  for every $1\le k \le d$ we need to show that 
		\[
		\E[X^{(\bi)}_k] \in L_{i_k,m}.
		\]
		
		Denote by $\underline X_k$ the collection of all variables $\{ X_{i,j}\}_{1\le i \le n, 1\le j \le k}$. Notice that our sorting algorithm is such that {\it given} $\underline X_k$ we can apply steps 1 through k above and thus the values of $X_j^{(i_1,\ldots, i_k)}$ for all $1\le i_1,\ldots,i_{k-1}\le m$, $1\le i_k \le m^{d-k+1}$, and $1\le j \le d$ are known.
		
		Next, fix $1\le k \le d$, and $1\le i_1, \ldots, i_k \le m$. Recall that given $\underline X_k$, 
		the set of $m^{d-k}$ variables $\l\{ X_k^{(i_1,\ldots, i_k, i_{k+1},\ldots, i_d)}\r\}_{1\le i_{k+1},\ldots, i_d \le m}$ is retrieved from the set
		\[
		 \l\{ X_k^{(i_1,\ldots, i_{k-1}, (i_k-1) m^{d-k}+1)}, \ldots, X_k^{(i_1,\ldots, i_{k-1}, i_km^{d-k})}  \r\}
		 \]
		  by a sequence of random permutations (given in steps $k+1,\ldots,d$) where all the permutations are determined by independent sequences of i.i.d.\ variables. Therefore, each individual variable $X_k^{(i_1,\ldots, i_k, i_{k+1},\ldots, i_d)}$ can take the value of any of the variables $X_k^{(i_1,\ldots, i_{k-1}, (i_k-1) m^{d-k}+j)}$
		for $j=1,\ldots, m^{d-k}$, with equal probability. Thus,
		\begin{equation}\label{eqn:mean_Xk}
		\E[X_k^{(\bi)}] = \E\l[\E\l[ X^{(\bi)}_k \given \underline X_k\r] \r] = \frac{1}{m^{d-k}} \sum_{j=1}^{m^{d-k}} \E[X_k^{(i_1,\ldots, i_{k-1}, (i_k-1) m^{d-k}+j)}].
		\end{equation}
		Next, recall that, as described in step k,   $X_k^{(i_1\ldots, i_{k-1},1)},\ldots, X_k^{(i_1,\ldots, i_{k-1}, m^{d-k+1})}$ are the order statistics of $m^{d-k+1}$ i.i.d.\ variables, uniform in $[-1,1]$. Therefore,
		\[
		\E\l[ X_k^{(i_1\ldots, i_{k-1},j)} \r] = -1 + \frac{2j}{m^{d-k+1}+1}.
		\]
		Putting this into \eqref{eqn:mean_Xk}, we have
		\[
		\E[X_k^{(\bi)}] = -1 + \frac{2i_km^{d-k}-m^{d-k}+1}{m^{d-k+1}+1}.
		\]
		All that remains to verify tis hat the last value is indeed in $L_{i_k,m}$, and this easy step completes the proof.
	\end{proof}

	
	For the next step, take $\mathcal G_1^\varepsilon$ and $\mathcal G_2^\varepsilon$  as in \eqref{eqn:G_12}.
	
	\begin{lemma}
		Almost surely, there exists (random) $N_\varepsilon>0$ such that, for all $n\ge N_\varepsilon$,  the following two statements are true:
		\begin{enumerate}
			\item If $(\bi,\bj)\in {\mathcal G}_1^\varepsilon$ then $h_r(X^{(\bi)},X^{(\bj)})=1$.
			\item If $(\bi,\bj)\in {\mathcal G}_2^\varepsilon$ then $h_r(X^{(\bi)},X^{(\bj)})=0$.
		\end{enumerate}
	\end{lemma}
	
	\begin{proof}
		Using Lemma \ref{lem:unif_bound}, the proof is identical to that of Lemma \ref{lem:goodset2}.
	\end{proof}
	
	\begin{lemma}\label{lem:hn}
		Let $H_r$ and $H_{n,r}$ be as defined above in \eqref{eqn:H} and \eqref{eqn:H_n_d}. Then, almost surely,
		\begin{align*}
		\limninf	\sup_{x\in B^d}\l| H_{n,r}(x)-H_r(x)\r| = 0.
		\end{align*}
	\end{lemma}
	
	\begin{proof}
		The proof here is identical to that of Lemma \ref{lem:approx}.
	\end{proof}
	
	\begin{lemma}\label{lem:converge_d}
		Let $0<r<2$, and $K_{n,r}$ be as defined above. Then $K_{n,r}\to K_r^{d}$ with respect to the cut-norm, almost surely, as $n\to\infty$.
	\end{lemma}
	
	\begin{proof}
		The proof is similar to that of Lemma \ref{lem:convergencekernel2}, and we will only highlight the required updates. We use  similar notation as in the proof of Lemma \ref{lem:convergencekernel2}. 
		Therefore, \eqref{K_n,r-K_r^1_estimation=2} is replaced by
		\begin{equation}\label{K_n,r-K_r^1_estimation2}
		|K_{n,r}(x,y)-K_r^2(x,y)|\leq \begin{cases}
		\frac{2^{3(d+1)}\varepsilon}{3r^{3d}} & ((\bi(x),\bi(y))\in \G_1^\varepsilon\\
		0 & ((\bi(x),\bi(y))\in \G_2^\varepsilon\\
		\frac{3\cdot 2^d}{r^d}& ((\bi(x),\bi(y))\in \mathcal B^\varepsilon\\			 
		\end{cases}\,,
		\end{equation}
		and so 
		\begin{equation}\label{eq:eavalvsvs2}
		\begin{aligned}
		\int_{B^2\times B^2}|K_{n,r}(x,y)-K_r^2(x,y)|d\nu_2(x)d\nu_2(y) &\leq \frac{2^{3(d+1)}\varepsilon}{3r^{3d}} \cdot \frac{|\G_1^\varepsilon|}{n^2} + \frac{3\cdot 2^d}{r^d}\cdot \frac{|\mathcal{B}^\varepsilon|}{n^2}\\
		&\leq \frac{2^{3(d+1)}\varepsilon}{3r^{3d}}+ \frac{3\cdot 2^d}{r^d}\cdot \frac{|\mathcal{B}^\varepsilon|}{n^2}.
		\end{aligned}
		\end{equation}
		In addition, we have, for $\varepsilon>0$, 
		\begin{align*}
		|\mathcal B^{\varepsilon}|\le 2d (2-r)\varepsilon n^2.
		\end{align*}
Substituting the last bound into \eqref{eq:eavalvsvs2} shows that, for all $n\geq N_\varepsilon$,
\[
\int_{B^2\times B^2}|K_{n,r}(x,y)-K_r^2(x,y)|d\nu_2(x)d\nu_2(y) \leq \frac{2^{3(d+1)}\varepsilon}{3r^{3d}}+ \frac{6d(2-r)2^d}{r^d}\varepsilon\,.
\]
		Observe that in the special case  $d=2$ we obtain  the bounds derived in the proof of Lemma \ref{lem:convergencekernel2}.  Since $\varepsilon>0$ is arbitrary, we are done.
	\end{proof}

	\section{Proofs of Theorems \ref{thm:limiteigenvalue} and \ref{thm:1to2}}\label{sec:thm1}

	In this section we finally complete the proofs of Theorems \ref{thm:limiteigenvalue} and \ref{thm:1to2} using the eigenvalue interlacing theorem,  see Theorem 4.3.28 in \cite{horn}.
	
	\begin{thm}[Eigenvalue Interlacing Theorem]\label{ft} Suppose $A$ is a real symmetric $n \times n$ matrix. Let $B$ be a $m\times m$ principal submatrix (obtained by deleting both the $i$-th row and the $i$-th column for some values of $i$). Suppose $A$ has eigenvalues $\alpha_{1}\ge\cdots\ge \alpha_n$ and $B$ has eigenvalues $\beta_1\ge \cdots \ge \beta_m$. Then, for every $1\leq k\leq m$
		\begin{align*}
		\alpha_{k+n-m}\le \beta_k\le \alpha_k\,.
		\end{align*}
	\end{thm}
	
	\begin{proof}[Proof of Theorem \ref{thm:limiteigenvalue}]
		The case of $d=1$ follows from the discussion in Sections \ref{sec:eigenvalue} and \ref{sec:d1} . In fact, from Lemma \ref{lem:equaleigenvalues}, we have that $\spec(W_{n,r}) = \spec(\K_{n,r})$, from Lemma \ref{lem:kernelconvergence} we have that $K_{n,r} \to K^1_r$ almost surely in cut norm, and, from Lemma \ref{lem:alleigenvalues}, we have that,   with the exception of the eigenvalues  $1/2$ and $1$,  all eigenvalues of $\K_r^1$ lie in $(-0.3,0.3)$. Finally, applying  Lemma \ref{lem:szegedy} proves the result. 
		
		For $d\ge 2$, using Lemmas \ref{lem:spec_d},\ref{lem:converge_d},  \ref{lem:alleigenvalues} and Lemma \ref{lem:szegedy} implies the result for all $n=m^d$. We are left to prove that the statement holds for any sequence of $n$.
		
		Suppose that $n>0$ is not in the form $n=m^d$. Then there exists $m>0$ such that
		$(m-1)^d< n < m^d$. Let $\lambda_{1,n}\ge \cdots \ge \lambda_{n,n}$ be the eigenvalues of $W_{n,1}$. Then Theorem \ref{ft} implies that 
		\begin{align*}
		\{j \; :\; |\lambda_{j,(m-1)^d}|>\lambda\}\subseteq \{j \; :\; |\lambda_{j,n}|>\lambda\}\subseteq  \{j \; :\; |\lambda_{j,m^d}|>\lambda\}.
		\end{align*}
		Taking $m\to \infty$, and using the convergence of the eigenvalues for $n=m^d$, concludes the proof.	 
	\end{proof}

	\begin{proof}[Proof of Theorem \ref{thm:1to2}]
		Combine Lemmas \ref{lem:k_L},  \ref{lem:spec_d}, \ref{lem:converge_d},  and \ref{lem:szegedy}   as in the proof of Theorem \ref{thm:limiteigenvalue}. 		
		\end{proof}
	
	\begin{proof}[Proof of Theorem \ref{thm:0to1}]
		
		Combine Lemmas \ref{lem:r<1}, \ref{lem:spec_d}, \ref{lem:converge_d} and \ref{lem:szegedy}, as in the proof of Theorem \ref{thm:limiteigenvalue}.
	\end{proof}

\begin{proof}[Proof of Corollary \ref{cor}]
	Recall that $\gamma_2^{(n)}=1-\lambda_2^{(n)}$, where $\lambda_2^{(n)}$ is the second largest eigenvalue of $W_{n,r}$. The proof is in three parts, one for each of the claims of the Corollary.
	
\paragraph{\bf Proof of first claim.} Let $r=1$.  Theorem \ref{thm:limiteigenvalue}-(2) implies that, for every $\epsilon>0$, almost surely there exists $N>0$ such that, for all $n\ge N$,
\begin{align*}
\lambda_2^{(n)}\in \l(\frac{1}{2}-\eps,\, \frac{1}{2}+\eps\r).
\end{align*}
Hence the result, as $\eps$ is arbitrary and $\gamma_2^{(n)}=1-\lambda_2^{(n)}$.
	
\paragraph {\bf Proof of second claim.} Let $r\in (1,2)$. Let $\lambda_2$ be the second largest eigenvalue of $\mathcal K_r^d$. 
Lemmas \ref{lem:k_L} and \ref{lem:lessthan1} imply that there exists $\epsilon>0$ such that
	\begin{align*}
	\eps < \lambda_2<\frac{1}{2}-\epsilon.
	\end{align*}
	Therefore, by Lemmas \ref{lem:spec_d}, \ref{lem:converge_d} and \ref{lem:szegedy}, almost surely there exists $N>0$ such that, for all $n\ge N$,
	\begin{align}		\frac{\eps}{2} \le \lambda_2^{(n)}\le \frac{1}{2}-\frac{\epsilon}{2}.
	\end{align} 
This implies that, for $n \ge N$,
	\begin{align*}
	\frac{1}{2}+\frac{\epsilon}{2}\le \gamma_2^{(n)}\le 1-\frac{\epsilon}{2}.
	\end{align*}
	Hence the result.

	\paragraph {\bf Proof of third claim.} Let $r\in (0,1)$. Lemmas \ref{lem:r<1} and \ref{lem:lessthan1} imply that there exists $\epsilon>0$ such that 
		\begin{align*}
			\frac{1}{2} +{\eps} <  \lambda_2 < 1-\eps.
		\end{align*}
	By Lemmas  \ref{lem:spec_d}, \ref{lem:converge_d} and \ref{lem:szegedy}, almost surely there exists $N_3>0$ such that, for all $n>N$, 
	\begin{align}
		\frac{1}{2} + \frac{\eps}{2} \le \lambda_2^{(n)}\le 1-\frac{\eps}{2},
	\end{align}
	which implies that 
		\begin{align*}
		\frac{\epsilon}{2} \le \gamma_2^{(n)}\le \frac{1}{2}-\frac{\epsilon}{2}.
	\end{align*}
	Hence the result.
\end{proof}

\section{Conclusion}
We have  shown that, almost surely, the second largest eigenvalue of $W_{n,r}$ is larger (smaller) than $1/2$ if $0<r<1$ (respectively, $1<r<2$) for all large $n$. We also proved that, if $r=1$, then $W_{n,r}$ has at least $\binom{d}{k}$ many eigenvalues around $1/2^k$. In Section \ref{sec:eigenvalue}, in order to study  the eigenvalues of $W_{n,r}$, we studied the eigenvalues of the limiting operator $\K_r^d$ . We proved that $\K_r^d$ is a self-adjoint and compact operator with the largest eigenvalue $1$, and the second largest eigenvalue is larger (smaller) than $1/2$ for $0<r<1$ (respectively, $1<r<2$). We conjecture that  the second largest eigenvalue of $\K_r^d$ is both continuous and monotonically decreasing in $0<r<2$.

In the above discussion two vertices in the graph are connected if they lie in a cube of side-length $r$. 
We note that our results can be extended to the case where the cube is replaced by general box. More precisely, let $r_1,\ldots, r_d\in (0,2)$. Define, for $x,y\in [-1,1]^d$, 
\[
h_{r_1,\ldots, r_d}(x,y)=\prod_{i=1}^d\one_{\{|x_i-y_i|\le r_i\}}.
\] 
Let $G_n$ be a random graph with $n$ points $\{X_1,\ldots, X_n\}$, where $X_1,\ldots, X_n$  are i.i.d. uniformly distributed random variables in $[-1,1]^d$, such that two vertices  $X_i, X_j$ are connected if, and only if, $h_{r_1,\ldots, r_d}(X_i, X_j)=1$. Let $A_n=(a_{ij})=(h_{r_1,\ldots, r_d}(X_i, X_j))$ be the adjacency matrix of $G_n$. Define 
\[
W_n=D_n^{-\smallhalf}A_nD_n^{\smallhalf},
\]
where $D_n=\mbox{diag}(d_1,\ldots, d_n)$ with $d_i=\sum_{k=1}^{n}a_{ik}$. Then it can be shown that the second largest eigenvalue of $W_n$ is almost surely smaller (larger) than $1/2$ when $r_1,\ldots, r_d\in (1,2)$ (respectively, if $r_i\in (0,1)$ for some $1\le i\le d$) for all large $n$. In order to prove this claim one needs to study the eigenvalues of the integral kernel operator $\K_{r_1,\ldots, r_d}$ in $L^2([-1,1]^d,\nu_d)$ with kernel
\[
K_{r_1,\ldots, r_d}(x,y)=\prod_{i=1}^d\frac{h_{r_i}(x_i,y_i)}{\sqrt{H_{r_i}(x_i)H_{r_i}(y_i)}},\;\; x,y\in [-1,1]^d.
\]
Let $(\lambda_{i,k})_{k\in \N}$, for $i=1,\ldots, d$, be the eigenvalues of $\K_{r_i}$, where $\K_{r_i}$ is an integral kernel operator in $L^2([-1,1], \nu_1)$ with respect to the kernel $$K_{r_i}(x,y)=\frac{h_{r_i}(x,y)}{\sqrt{H_{r_i}(x)H_{r_i}(y)}}, \; \; x,y\in [-1,1].
$$
Then, following  the proof of Lemma \ref{lem:alleigenvalues}, it can be shown that $ (\lambda_{1,k_1}\cdots \lambda_{d,k_d})_{k_1,\ldots, k_d\in \N}$ are the eigenvalues of $\K_{r_1,\ldots, r_d}$. As a consequence, from Lemma \ref{lem:r<1} it follows that the second largest eigenvalue of $\K_{r_1,\ldots, r_d}$ is larger than $1/2$ if $r_i\in (0,1)$ for some $i\in \{1,\ldots, d\}$, and from Lemma \ref{lem:secondevforL} it follows that all the eigenvalues (except $1$) of $\K_{r_1,\ldots, r_d}$ are strictly smaller than $1/2$ when $r_1,\ldots, r_d\in (1,2)$.

Finally, we considered here only the $L^{\infty}$ norm, which made the details of the calculations easier. We conjecture that {qualitatively} similar results  should be true if we replace the $L^{\infty}$-norm by other norms, including   $L^2$.

	\appendix
	
	\section{Proofs of Lemmas \ref{lem:concentration}, \ref{lem:concentration2}, \ref{lem:unif_bound}}\label{app1}
	Let $U_1,\ldots, U_n$ be i.i.d.\ uniformly distributed random variables in $[0,1]$, and  let $U^{(1)}, \ldots, U^{(n)}$ be their order statistics, i.e., $U^{(1)}\le  \cdots\le  U^{(n)}$. It is well known that the $k$-th order statistics is a beta random variable, or more precisely,
	\begin{align*}
	U^{(k)}\sim \mbox{Beta}(k,n+1-k),
	\end{align*}
	which implies that
	\[
	\E[U^{(k)}] = \frac{k}{n+1}.
	\]
	In our situation, we have $X_1,\ldots, X_n$ i.i.d.\ uniformly distributed in $[-1,1]$. Since we can write $X_i = 2U_i -1$, with $U_i$ as above, then  the order statistics $X^{(1)}\le\cdots \le X^{(n)}$ can also be written as $X^{(k)} = 2U^{(k)}-1$. Thus, we have
	\begin{equation}\label{eqn:mean}
	\E[X^{(k)}] = -1 + \frac{2k}{n+1}.
	\end{equation} 
	
	To prove the lemmas we use the sub-Gaussian property of the beta distribution. A random variable $X$ with  finite mean $\mu=\E [X]$ is said to be {\it sub-Gaussian} if there is a $\sigma>0$ such that, for all $\lambda\in \R$,
	\begin{align}\label{eqn:sub_gauss}
	\E[e^{\lambda(X-\mu)}]\le e^{\frac{\lambda^2\sigma^2}{2}}.
	\end{align}
	The constant $\sigma^2$ is called a {\it proxy variance}, and we say that $X$ is $\sigma^2$ sub-Gaussian. 
	
	Let  $X$ be a $\sigma^2$ sub-Gaussian random variable. Then Markov's inequality, together with \eqref{eqn:sub_gauss}, implies that, for any $\lambda, t>0$,
	\begin{align*}
	\P[X-\mu>t]=\P[e^{\lambda(X-\mu)}>e^{\lambda t}]\le e^{-\lambda t+\frac{\lambda^2\sigma^2}{2}}.
	\end{align*}
	Optimizing the upper bound over $\lambda$ yields,
	\begin{align*}
	\P[X-\mu>t]\le e^{-\frac{t^2}{2\sigma^2}}.
	\end{align*}
	Similarly, it can be shown that if $X$ is $\sigma^2$ sub-Gaussian, then, for all $t>0$
	\begin{align*}
	\P[X-\mu<-t]\le e^{-\frac{t^2}{2\sigma^2}}.
	\end{align*}
	Therefore, we conclude that  if $X$ is $\sigma^2$ sub-Gaussian, then, for all $t>0$,
	\begin{align}\label{eqn:chernoff}
	\P[|X-\mu|>t]\le 2e^{-\frac{t^2}{2\sigma^2}}.
	\end{align}
	
	To prove Lemma \ref{lem:concentration} we will use the following result.
	
	\begin{theorem}[Theorem 1 in \cite{arbel}]\label{re:beta}
		The Beta$(\alpha, \beta)$ distribution is  $(4(\alpha+\beta+1))^{-1}$ sub-Gaussian.
	\end{theorem}
	
	We can now prove Lemma \ref{lem:concentration}.
	\begin{proof}[Proof of Lemma \ref{lem:concentration}]
		For any $\delta>0$, we have
		\begin{align*}
		\P\l[\l|X^{(k)}-\E[X^{(k)}]\r|>\delta \r]=	\P\l[\l|U^{(k)}-\E[U^{(k)}]\r|>\frac{\delta}{2} \r].
		\end{align*}
		Recall that $U^{(k)}\sim$Beta$(k,n-k+1)$. Theorem \ref{re:beta} implies that for all $k=1,\ldots, n$, $U^{(k)}$ is sub-Gaussian with $\sigma^2 = (4(n+2))^{-1}$. Therefore, applying \eqref{eqn:chernoff}, we have
		\begin{align*}
		\P\l[\l|X^{(k)}-\E[X^{(k)}]\r|>\frac{1}{n^{1/3}} \r]\le 2e^{-\frac{n^{1/3}}{2}}.
		\end{align*}
		Using the union bound, we have 
		\begin{align*}
		\P\l[\bigcup_{k=1}^n\l\{\l|X^{(k)}-\E[X^{(k)}]\r|>\frac{1}{n^{1/3}}\r\}\r]\le 2n e^{-\frac{n^{1/3}}{2}}.
		\end{align*}
		Thus,
		\begin{align*}
		\sum_{n=1}^{\infty}\P\l[\bigcup_{k=1}^n\l\{\l|X^{(k)}-\E[X^{(k)}]\r|>\frac{1}{n^{1/3}}\r\}\r]<\infty,
		\end{align*}
		and  the result follows from the Borel-Cantelli lemma.
	\end{proof}

		\begin{proof}[Proof of Lemma \ref{lem:concentration2}]
		Let $X^{(\bi)}=(X^{(\bi)}_1,X^{(\bi)}_2)$.
		Since 
		\[
		\sup_{\bi}\l\|X^{(\bi)}-\E[X^{(\bi)}]\r\| = \max_{j=1,2}\sup_{\bi}\l|X_j^{(\bi)}-\E[X_j^{(\bi)}]\r|,
		\]
		we will bound each of the coordinates separately.
		
		We start with $X_1^{(\bi)}$. Recall that $X_1^{(1)}, \ldots, X_1^{(n)}$ are the order statistics of $X_{1,1},\ldots, X_{n,1}$. Using Lemma \ref{lem:concentration}, almost surely there exists $N_1>0$ such that, for  $n\ge N_1$,
		\begin{align*}
		\sup_k\l|X_1^{(k)}-\E[X_1^{(k)}]\r|\le  \frac{1}{n^{1/3}}.
		\end{align*}
		Denote by $E$ the almost-sure event described above and fix $\omega\in E$.
		Next, fix $p$, and recall that $X_1^{(p,1)},\ldots, X_1^{(p,m)}$ is a permutation of $X_1^{((p-1)m + 1)}, \ldots, X_1^{(pm)}$, which implies that, for all $q=1,\ldots, m$, we have that $X_1^{(p,q)} = X_1^{((p-1)m + r)}$, for some $1\le r \le m$.
		Thus,
		\[
		\l| X_1^{(p,q)} - \E[X_1^{(p,q)}]\r| \le \l| X_1^{((p-1)m + r)} - \E[X_1^{((p-1)m + r)}]\r|+ \l|\E[X_1^{((p-1)m + r)}]-\E[X_1^{(p,q)}]\r|.
		\]
		Next, recall that $\E[X_1^{((p-1)m + r)}]=-1 +\frac{2((p-1)m+r)}{n+1}$ (cf.\ \eqref{eqn:mean}), and $\E[X_1^{(p,q)}]= -1+\frac{2p m - m +1}{n+1}$ (cf.\ \eqref{eqn:E_X_pq}).
		Therefore, for $n\ge N_1$, we have 
		\begin{align*}
		\l| X_1^{(p,q)} - \E[X_1^{(p,q)}]\r|  \le &  \frac{1}{n^{1/3}}+\frac{1}{ m}\le \frac{1}{n^{1/6}}.
		\end{align*}
		Since this is true for all $p,q$, we have, for all $n\ge N_1$,
		\begin{align}\label{eqn:b-c1}
		\sup_{\bi}\l|X^{(\bi)}_1 - \E[X^{(\bi)}_1]\r|\le \frac{1}{n^{1/6}}.
		\end{align}
		
		We proceed with bounding $X_2^{(\bi)}$. Suppose that $\underline X_1$ is given. Then, for every $p$, we have that $X_2^{(p,q)}$  is the $q$-th order statistic of $m$ i.i.d.\ uniform random variables in $[-1,1]$. 
		By Lemma \ref{lem:concentration}, for every $p$, almost surely there exists $N_2(p) > 0$ such that, for all $n\ge N_2(p)$, we have 
		\[
		\sup_{k=1,\ldots,m}|X_2^{(p,k)}-\E[X_2^{(p,k)}]|\le \frac{1}{m^{1/3}}=\frac{1}{n^{1/6}}.
		\]
		Taking $N_2 = \max_{1\le p \le m} N_2(p)$, then for $n\ge N_2$ we have
		\begin{align}\label{eqn:b-c2}
		\sup_{\bi}|X_2^{(\bi)}-\E[X_2^{(\bi)}]|\le \frac{1}{m^{1/3}}=\frac{1}{n^{1/6}}.
		\end{align}
		
		To conclude, we showed that almost surely there exists $N = \max(N_1, N_2)$ such that, 
		for all $n\ge N$, both
		\eqref{eqn:b-c1} and  \eqref{eqn:b-c2} hold. This concludes the proof.
	\end{proof}

	\begin{proof}[Proof of Lemma \ref{lem:unif_bound}]
	Fix $1\le k\le d$, and recall that, for every $i_1,\ldots, i_{k-1}$, in Step k of our construction we had that $X_{k}^{(i_1,\ldots,i_{k-1},1)}\le \cdots\le X_{k}^{(i_1,\ldots,i_{k-1},m^{d-k+1})}$ are the order statistics of $m^{d-k+1}$ i.i.d.\ uniformly distributed random variables in $[-1,1]$.  Therefore, by Lemma \ref{lem:concentration}, almost surely there exists $N_k(i_1,\ldots, i_{k-1})$ such that, for all $n\ge N_k(i_1,\ldots, i_{k-1})$, we have
	\begin{align*}
	\sup_{j=1,\ldots,m^{d-k+1}}|X_{k}^{(i_1,\ldots,i_{k-1},j)}-\E[X_{k}^{(i_1,\ldots,i_{k-1},j)}]|\le \frac{1}{m^{(d-k+1)/3}}.
	\end{align*}
	Next, fix $\bi = (i_1,\ldots, i_d)$, and let $E$ be the almost sure event above. Fix $\omega\in E$ and suppose that $n\ge N_k(i_1,\ldots, i_{k-1})$.
	Recall that the variable  $X_{k}^{(\bi)}$ is equal to  one of the variables $X_{k}^{(i_1,\ldots,i_{k-1},(i_{k}-1)m^{d-k}+1)}, \ldots,  X_{k}^{(i_1,\ldots,i_{k-1},i_{k}m^{d-k})}$.
	Let  $r = r(\omega)$ be such that $X_{k}^{(i_1,\ldots, i_d)} =$ $ X_{k}^{(i_1,\ldots,i_{k-1},(i_{k}-1)m^{d-k}+r)}$. Then,
	\begin{align*}
	\l|X_{k}^{(\bi)}-\E[X_{k}^{(\bi)}]\r|\le & \l|X_{k}^{(i_1,\ldots,i_{k-1},(i_{k}-1)m^{d-k}+r)}-\E[X_{k}^{(i_1,\ldots,i_{k-1},(i_{k}-1)m^{d-k}+r)}]\r|
	\\&\qquad\qquad+ \l|\E[X_{k}^{(i_1,\ldots,i_{k-1},(i_{k}-1)m^{d-k}+r)}]-\E[X_{k}^{(i_1,\ldots,i_{d})}]\r|
	\end{align*}
	Since both $\E[X_{k}^{(i_1,\ldots,i_{k-1},(i_{k}-1)m^{d-k}+r)}]$ and $\E[X_{k}^{(i_1,\ldots,i_{d})}]$ lie in $L_{i_{k},m}$, we have 
	\begin{align*}
	\l|\E[X_{k}^{(i_1,\ldots,i_{k-1},(i_{k}-1)m^{d-k}+r)}]-\E[X_{k}^{(i_1,\ldots,i_{d})}]\r|\le \frac{1}{m}.
	\end{align*}
	In addition, since we assume $n\ge N_k(i_1,\ldots,i_{k-1})$, we have
	\[
	\l|X_{k}^{(i_1,\ldots,i_{k-1},(i_{k}-1)m^{d-k}+r)}-\E[X_{k}^{(i_1,\ldots,i_{k-1},(i_{k}-1)m^{d-k}+r)}]\r|\le \frac{1}{m^{(d-k+1)/3}},
	\]
	and therefore,
	\[
	\l|X_{k}^{(\bi)}-\E[X_{k}^{(\bi)}]\r|\le  \frac{1}{m^{(d-k+1)/3}}+\frac{1}{m}\le \frac{1}{n^{1/3d}}.
	\]
	Taking $N_k = \max_{i_1,\ldots,i_{k-1}} N_k(i_1,\ldots,i_{k-1})$ and $n\ge N_k$, we  have
	\begin{equation}\label{eqn:bound_k}
	\sup_{\bi}\l|X_{k}^{(\bi)}-\E[X_{k}^{(\bi)}]\r|\le \frac{1}{n^{1/3d}}.
	\end{equation}
	Finally, let $N=\max\{N_k\; : \; 1\le k\le d\}$. 
	Then \eqref{eqn:bound_k} holds for all $n\ge N$, and we are done.
\end{proof}

	\section{Integral kernel operators}\label{sec:appendix_spec}
	%
	%

	We now  provide a proof for Lemma \ref{lem:szegedy}, which extends Lemma 1.11 in \cite{szegedy11}. The proof will make use of two lemmas. 
	
	Recall that $\mathcal H$ denotes the Hilbert space $L^2(V,\nu)$. A sequence  $\{f_n\}_{n=1}^{\infty}$ is called {\it weakly convergent} if $\{\langle f_n,g\rangle \}_{n=1}^{\infty}$  converges for every $g\in \mathcal H$.
	
	\begin{lemma}[Lemma 1.10 in \cite{szegedy11}]\label{lem:l2convergence}
		Let $K$ be the cut norm limit of $\{K_n\}_{n=1}^{\infty}$. Let $\{f_n\}_{n=1}^{\infty}$ be a weakly convergent sequence in $\mathcal H$ with limit $f$ such that $\|f_n\|_2=1$ for every $n$ and $\mathcal K_nf_n=\lambda_n f_n$, where $\lim_{n\to \infty}\lambda_n=\lambda\neq 0$. Then $\{f_n\}_{n=1}^{\infty}$ converges in $L_2$ to $f$ and $\mathcal K f=\lambda f$.
	\end{lemma}
	
	\begin{proof}[Proof of Lemma \ref{lem:szegedy}]
		 Let $\{\lambda_{n,j}\}_{j=1}^\infty$ be the eigenvalues of $\K_n$, listed with multiplicities. 
		If $\K_n$ is a finite rank operator then we put an infinite number of zeroes at the end.	We assume that $\{|\lambda_{n,j}|\}_{j=1}^\infty$ is a  decreasing sequence. Since $K_n$ is  symmetric, using  the spectral decomposition theorem for $\K_n$, the kernel function $K_n$ can be expressed as 
		\begin{align*}
		K_n(x,y)=\sum_{j=1}^{\infty}\lambda_{n,j}\varphi_{n,j}(x)\varphi_{n,j}(y),
		\end{align*}
		where $\{\varphi_{n,j}\}_{\lambda_{n,j}\ne 0}$ is an orthonormal system in $\mathcal H$. 
		For $\lambda_{n,j}=0$ we take $\varphi_{n,j}$ to be an arbitrarily chosen function of unit length.
		
		Note that 
		\begin{align*}
		\iint |K_n(x,y)|^2d\nu(x)d\nu(y)=\sum_{j=1}^{\infty}\lambda_{n,j}^2,
		\end{align*}
		and since we assume  that $\|K_n\|_{\infty}\le C$ we also have that
		\begin{align*}
		\sum_{j=1}^{\infty}\lambda_{n,j}^2\le C.
		\end{align*}
		For every $j$, $\{\lambda_{n,j} \}_{n=1}^{\infty}$ is bounded. In addition,  every bounded  sequence  in a Hilbert space contains a weakly convergent subsequence. Therefore, for every fixed $j$ we can find  a subsequence $\{n_i\}_{i=1}^{\infty}$ such that $\{\varphi_{n_i,j}\}_{i=1}^{\infty}$ is weakly convergent in $\mathcal H$ and $\{\lambda_{n_i,j}\}_{i=1}^{\infty}$ is convergent. Let $\varphi_j$ be the weak limit of $\{\varphi_{n_i,j}\}_{i=1}^{\infty}$ and $\lambda_j$ be the limit of $\{\lambda_{n_i,j}\}_{i=1}^{\infty}$. Then it can be shown that (a) $\sum_{j=1}^{\infty}\lambda_{j}^2\le C$, and (b)  $\{|\lambda_j| \}_{j=1}^{\infty}$ is a decreasing sequence.
		These two facts  imply that $|\lambda_j|\le \frac{\sqrt C}{\sqrt j}$ for every $j$. Indeed, for $j_0\in \N$,
		\begin{align}\label{eqn:eigenvalue_bound}
		\sum_{j=1}^{j_0}\lambda_{j}^2\le C \imply |\lambda_{j_0}|\le \sqrt{\frac{C}{j_0}}.
		\end{align}
		If $\lambda_j\neq 0$, then Lemma \ref{lem:l2convergence} implies that 
		\begin{align}\label{eqn:re1}
		\lim_{i\to \infty} \int |\varphi_{n_i,j}(x)-\varphi_j(x)|^2d\nu(x)\to 0.
		\end{align}
		Using the triangle and Cauchy-Schwarz inequalities gives that, for $\lambda_{j_1}, \lambda_{j_2}\neq 0$,  we have
		\begin{align*}
		\lim_{i\to \infty}|\langle \varphi_{n_i,j_1},\varphi_{n_i,j_2}\rangle-\langle \varphi_{j_1},\varphi_{j_2}\rangle|=0.
		\end{align*}
		This implies that, if $\lambda_{j_1}$ and $\lambda_{j_2}$ are non-zero, then $\langle \varphi_{j_1},\varphi_{j_2}\rangle=0$.
		Therefore $\{\varphi_j\}_{j: \lambda_{j}\neq 0}$ is an orthogonal system of functions. Define 
		\begin{align*}
		K'(x,y)=\sum_{j: \lambda_j\neq 0}\lambda_j\varphi_j(x)\varphi_j(y).
		\end{align*}
		We show that $K(x,y)=K'(x,y)$. Let $t\in \N$. Then 
		\begin{equation}\label{eqn:diff_norm}
		\begin{split}
		\|K_{n_i}-K'\|_{\square}&\le 	\l\|\sum_{j=1}^t(\lambda_{n_i,j}\varphi_{n_i,j}(x)\varphi_{n_i,j}(y)-\lambda_j\varphi_j(x)\varphi_j(y)\r\|_{\square}
		\\&\qquad\qquad+\l\|\sum_{j=t+1}^{\infty}\lambda_{n_i,j}\varphi_{n_i,j}(x)\varphi_{n_i,j}(y)\r\|_{\square}+\l\|\sum_{j=t+1}^{\infty}\lambda_j\varphi_j(x)\varphi_j(y)\r\|_{\square}.
		\end{split}
		\end{equation}
		To bound the first term in \eqref{eqn:diff_norm}, fix $1\le j\le t$ and 
		note that, since  $\lambda_{n_i,j}\to \lambda_j$, there exists a $i_0>0$ such that, for all $i>i_0$ we have,
		\begin{align}\label{eqn:re2}
		|\lambda_{n_i,j}-\lambda_j|\le \frac{1}{(t+1)^{\frac{3}{2}}}.
		\end{align}	
		Using  \eqref{eqn:re1}, \eqref{eqn:re2} and triangle inequality,  we have  for all $i>i_0$,
		\begin{align*}
		\l\|\sum_{j=1}^t(\lambda_{n_i,j}\varphi_{n_i,j}(x)\varphi_{n_i,j}(y)-\lambda_j\varphi_j(x)\varphi_j(y)\r\|_{\square}\le \frac{C'}{\sqrt{t+1}},
		\end{align*}
		for some positive constant $C'$.  
		
		To bound the other two terms in \eqref{eqn:diff_norm} we use the notion of {\it spectral radius}. Let $M$ be a self-adjoint kernel operator, and $\{\lambda_j(M)\}_{j=1}^\infty$ be its eigenvalues. {Then the spectral radius of $M$ is defined as}
		\[
		\mathrm{rad}(M) = \sup_{j} |\lambda_j(M)|.
		\]
		Lemma 1.5 in \cite{szegedy11} states that $\|M\|_{\square} \le \mathrm{rad}(M)$. Note that since the absolute eigenvalues are decreasing, and using \eqref{eqn:eigenvalue_bound}, for both the second and the third terms in \eqref{eqn:diff_norm} we have that the spectral radius is bounded by 
		$\sqrt{C}/\sqrt{t+1}$.
		
		To conclude, we can show that there exists $C''>0$ such that for a large enough $i$ we have,
		\begin{align*}
		\|K_{n_i}-K'\|_{\square}&\le\frac{C''}{\sqrt{t+1}}.
		\end{align*}
		Letting $t\to \infty$ we get that $K_{n_i}$ converges to $K'$ in the cut norm. Since we also know that $K_{n_i}\to K$, we conclude that $K'\equiv K$.
		
		Take $\lambda>0$ such that $\pm\lambda \not\in \spec(K)$, and let $t$ be an integer greater than $C\lambda^{-2}$. Note that, for $j>t$,
		\begin{align*}
		\lambda_{n_i,j}\le \frac{\sqrt C}{\sqrt j}\le \frac{\sqrt C}{\sqrt t}\le \lambda.
		\end{align*}
		Let $m_t=\min\{|\lambda-|\lambda_j|| \; :\; 1\le j\le t \}>0.$ Note that $m_t>0$,
		as $\lambda$ and $-\lambda$ are not eigenvalues of $K$.
		In addition, there exists $i_0>0$ such that, for all $i>i_0$ and  $1\le j \le t$,
		\begin{align*}
		|\lambda_{n_i,j}-\lambda_j|\le \frac{m_t}{2}.
		\end{align*}
		Therefore, we conclude that $\lambda_j>\lambda \Leftrightarrow \lambda_{n_i,j}>\lambda$ and $\lambda_j<-\lambda  \Leftrightarrow \lambda_{n_i,j}<-\lambda$, for $1\le j\le t$ as $i\to \infty$. Hence \eqref{eqn:szegedy} is true for $\{K_{n_i}\}$. That is,
		\begin{align*}
		\lim_{i\to \infty} |\{\spec(K_{n_i}) \cap (\lambda,\infty)\}| &=|\{\spec(K) \cap (\lambda,\infty)\}|,
		\\\lim_{i\to \infty} |\{\spec(K_{n_i}) \cap (-\infty,-\lambda)\}|&=|\{\spec(K) \cap (-\infty,-\lambda)\}|.
		\end{align*}
		To conclude, we have to show convergence when $n\to \infty$ (as opposed to $n_i\to\infty$).
		Suppose that we can choose an infinite subsequence such that  \eqref{eqn:szegedy} does not hold. This leads to an immediate contradiction, since from such a subsequence we cannot choose a subsequence which satisfies the result, and we are done.
	\end{proof}


	\bibliography{bibgeometry}
	\bibliographystyle{plain}
	
\end{document}